\documentclass[12pt]{article}
\usepackage{amsmath}
\usepackage{amssymb}
\usepackage{amsthm}
\usepackage{latexsym}
\usepackage{color,fancyhdr,lastpage,graphicx}
\usepackage{subfigure, epsfig, epsf}
\usepackage{xspace}
\usepackage{setspace}
\usepackage{bbm}

\graphicspath{{thesis/Figures/}}

\theoremstyle{plain}

\newtheorem{lemma}{Lemma}[section]
\newtheorem{defn}[lemma]{Definition}
\newtheorem{prop}[lemma]{Proposition}
\newtheorem{theorem}[lemma]{Theorem}
\newtheorem{proposition}[lemma]{Proposition}
\newtheorem{corollary}[lemma]{Corollary}
\theoremstyle{remark}

\def\bg{\begin{color}{blue}}
\def\br{\begin{color}{red}}
\def\eg{\end{color}}
\def\er{\end{color}}
\def\note{\footnote{\bg\comment\eg}}
\def\note{}
\def\PP{\mathbb{P}}
\def\ve{{\varepsilon}}
\def\le{\leqslant}

\def\ge{\geqslant}

\def\es{\emptyset}
\def\da{\downarrow}

\def\E{{\mathbb E}}
\def\O{{\Omega}}
\def\Q{{\mathbb Q}}
\def\R{{\mathbb R}}
\def\C{{\mathbb C}}
\def\N{{\mathbb N}}

\def\Z{{\mathbb Z}}

\def\a{{\alpha}}
\def\b{{\beta}}
\def\d{{\delta}}
\def\D{{\Delta}}
\def\t{{\tau}}

\def\g{{\gamma}}
\def\G{{\Gamma}}
\def\s{{\sigma}}
\def\l{{\lambda}}

\def\th{{\theta}}
\def\o{{\omega}}

\def\cS{{\cal S}}
\def\cL{{\cal L}}
\def\cR{{\cal R}}

\def\cE{{\cal E}}
\def\cF{{\cal F}}
\def\cH{{\cal H}}

\def\cD{{\cal D}}
\def\cZ{{\cal Z}}
\def\cW{{\cal W}}

\def\q{\quad}
\def\id{\operatorname{id}}
\def\cp{\operatorname{cap}}
\def\dom{\operatorname{dom}}

\def\<{\langle}
\def\>{\rangle}

\def\ua{\uparrow}
\def\sse{\subseteq}
\def\sm{\setminus}

\renewcommand\epsilon{\ve}
\begin{document}
\bibliographystyle{plain}

%	FRONT MATTER
\begin{center}
\LARGE \textbf{Planar aggregation and the coalescing Brownian flow}

\vspace{0.2in}

\large {\bfseries James Norris
\footnote{Statistical Laboratory, Centre for Mathematical Sciences,
  Wilberforce Road, Cambridge, CB3 0WB, UK }
\& Amanda Turner
\footnote{Department of Mathematics and Statistics,
Fylde College,
Lancaster University,
Lancaster,
LA1 4YF, UK}}

\vspace{0.2in}
\small
\today

\end{center}
\begin{abstract}
We study a scaling limit associated to a model of planar aggregation.
The model is obtained by composing certain independent random conformal maps.
The evolution of harmonic measure on the boundary of the cluster is shown to
converge to the coalescing Brownian flow.
\end{abstract}
\vspace{0.2in}

%	BODY OF ARTICLE
%\ct

\section{Introduction}
A simplified model of aggregation in two dimensions may be formulated as follows.
Let $K_0$ be the closed disc of radius $1$, with centre at the origin, and
let $P_1$ be another closed disc, of diameter $\d>0$ say,
tangent to $K_0$ at a point chosen uniformly at random.
Set $D_0=(\C\cup\{\infty\})\sm K_0$ and $D_1=(\C\cup\{\infty\})\sm K_1$,
where $K_1=K_0\cup P_1$.
Write $F_1$ for the unique conformal isomorphism $D_0\to D_1$ such that
$F_1(\infty)=\infty$ and $F_1'(\infty)>0$.
Suppose then that $F_2,F_3,\dots$ are independent and identically distributed
copies of $F_1$, and define
$$
D_n=F_1\circ\dots\circ F_n(D_0).
$$
The complementary sets $K_n=\C\sm D_n$ form an increasing family and may be
considered as a model of aggregation, where, for each $n\ge0$, we add at time $n+1$
a new particle $P_{n+1}=K_{n+1}\sm K_n$.
Note that $P_{n+1}$ is the image under the conformal map $F_1\circ\dots\circ F_n$
of a disc distributed as $P_1$ and independent of $K_n$.
Since harmonic measure is invariant under conformal maps,
conditional on $K_n$, the random point at which
$P_{n+1}$ is attached to $K_n$ is distributed on the boundary of $K_n$
according to the normalized harmonic measure from infinity. In this respect, the
model is appropriate for the aggregation of diffusive particles `coming in from infinity'.
We note however that the added particle is distorted in a way depending
on the current cluster $K_n$.

In this paper, we identify a scaling limit for the restrictions to the boundary of $D_0$
of the maps $(F_m\circ\dots\circ F_n)^{-1}$, for
$m\le n$, as $m,n\to\infty$ and $\d\to0$ in a suitable way.
This can be thought of as identifying the time-evolution of the total harmonic measure
carried by the various `fingers' of the growing cluster.
The limit is not sensitive to the shape of the first added particle: although we
specified discs in the description above, the same sort of limit applies more generally.
The limit object is the flow of coalescing Brownian motions on the circle, also
known as the Brownian web.

A brief review of related work is given next, followed by some illustrations of typical
clusters, for certain cases of the model. From Section \ref{LFC} on, we confine our attention
to the restrictions of maps to the boundary of $D_0$, generalizing at the same 
time to a natural class of L\' evy random flows on the circle. We first show
a weak convergence result
for the time-evolutions of finitely many points. In Section \ref{CBF} we
describe a new canonical space for the coalescing Brownian
flow.
Then, in Section \ref{SKOR}, a larger flow-space, of Skorokhod type,
is introduced. This is a complete separable metric space
suitable for the formulation of weak convergence of stochastic flows
which are not necessarily
continuous in time or space.
The convergence of L\' evy flows to the coalescing Brownian flow,
is shown in Section \ref{MR}. Time reversal of flows is discussed in Section
\ref{TR}. Then, in Section \ref{model}, 
convergence results for the aggregation model 
are deduced as corollaries of the convergence of L\' evy flows.
Some technical details are left to an Appendix, where we discuss, in
particular, the relation between our formulation of the coalescing Brownian flow
and the Brownian webs of Fontes et al. 

\section{Review of related work}

\subsection{Coalescing Brownian motions}
It is straightforward to define a finite family of (standard) Brownian motions 
in one dimension, with 
given space-time starting points, independent until they collide, and
coalescing on collision.
The possibility to extend such a model to a space-time continuum of starting
points was shown by Arratia in 1979 in his PhD thesis \cite{A79}, where the
model was considered as a limit object for coalescing random walks on the integers.
Subsequent work has been carried out by many people
including Harris \cite{harris}, who was interested in general coalescing
stochastic flows, and Piterbarg \cite{Piterbarg}, who showed that
Arratia's flow arises as a weak limit of rescaled isotropic
stochastic flows. Further properties were developed by T\'{o}th and Werner \cite{TW} in 1998,
who used the flow to contruct an object which they called
the `continuous true self-repelling motion'. 
Tsirelson \cite{Tsirelson} studied Arratia's flow, formulated in terms of $L^2$-spaces, rather than pathwise,
as an example in the general theory of stochastic flows. 
Fontes, Isopi,
Newman and Ravishankar \cite{FINR,FN} introduced the name `Brownian web'
in 2004,
to describe a number of new formulations of Arratia's flow, and gave further
characterization and convergence results. 
The paper \cite{FINR}
characterizes the Brownian web as a
random element of a space of compact collections of
paths with specified starting points. 

In this paper we follow most closely the viewpoint of T\'{o}th and Werner but lay greater
stress on the almost sure flow-type properties, formulating
Arratia's flow as a random variable in
a certain complete separable metric space of continuous weak flows. 
We show that there is a
unique Borel probability measure on this space with respect to which all
the $n$-point motions are coalescing Brownian motions, and that this 
measure is invariant under time reversal.
Moreover, we show that any sequence of random such flows, all of 
whose $n$-point motions
converge to coalescing Brownian motions, converges to the coalescing Brownian 
flow.  
The exact correspondence between our work and that in \cite{FINR} is discussed in the Appendix. 

\subsection{Planar random growth}
\label{dlasec}

Our motivation for looking at the coalescing Brownian flow arises from a
surprising connection with planar random growth processes.

The simplest sorts of planar growth process to formulate take values in finite
subsets of $\Z^2$, starting from a singleton at the origin, 
and grow by the successive addition of sites adjacent to the present cluster, which are
chosen according to some distribution determined by the present cluster.
In 1961, Eden \cite{Eden} introduced one such process, where the added site is simply chosen
uniformly from all adjacent sites. This has been 
considered as a model for the growth of bacterial cells or tissue cultures 
of cells that are constrained from moving.
In 1981, Witten and Sander \cite{W+S} put forward another such process,
known as diffusion-limited aggregation or DLA. Here, particles perform 
random walks `starting from infinity' 
until they reach a site adjacent to the cluster. This site is then added to the cluster.
DLA is considered as a model for the formation of aggregates by deposition,
such as soot particles.
A family of further processes, indexed by a parameter $\eta\in[0,1]$ and
called dielectric-breakdown models,
were discussed by Niemeyer et al. \cite{NPW}.
In these processes, the random walk hitting probabilities on
sites adjacent to the cluster are interpreted as an electric field $E$;
the probabilities of attachment for new sites are then chosen proportional to $E^\eta$.
Thus the case $\eta=0$ is the Eden model and the case $\eta=1$ is DLA.

The primary interest in these and other related processes is in the 
asymptotic behaviour of large clusters. Computational investigations reveal
the emergence of complex structures, of fractal type, which sometimes appear a good match
for observed phenomena. However, such
investigations are highly sensitive to model variations and there are few notable
mathematical results, with the exception of Kesten's 1987 
growth estimate \cite{Kesten} for DLA. Moreover, the
simulations appear to show that clusters depend on the fine
lattice structure at large scales. This suggests that lattice-based processes may never
properly model continuum phenomena.

In 1998, Hastings and Levitov \cite{HL} formulated a family of continuum growth models
in terms of sequences of iterated conformal maps, indexed by a parameter $\a\in[0,2]$. 
As in the model described in the Introduction, 
they identify the cluster after the arrival of $n$ particles with a conformal map 
$\Phi_n=F_1\circ\dots\circ F_n$ on $D_0$, where, conditional on $\Phi_n$, the map $F_{n+1}$ corresponds
to a particle attached at a random point $z_{n+1}$ on the unit circle, but now of radius $\d_{n+1}=\d_0|\Phi_n'(z_{n+1})|^{-\a}$.
Hastings and Levitov argue, by comparing local growth rates, that their model can be related to lattice 
dielectric-breakdown by setting $\eta=\alpha-1$. Further exploration of this relation is discussed in the 
survey paper by Bazant and Crowdy \cite{B+C}.

Carleson and Makarov \cite{C+M}, in 2001, obtained a growth estimate for a deterministic analogue 
of the DLA model. 
In 2005, Rohde and Zinsmeister \cite{RZ} considered
the case $\alpha=0$ in the Hastings--Levitov family. 
They established a scaling limit, in the case where the particle size is small (but fixed),
as the number of particles tends to infinity, and showed that the limit sets were one dimensional. They
also gave estimates for the dimension of the limit sets in the case of general $\alpha$,
and discussed limits of deterministic variants.

We also consider the case $\alpha=0$ but in the limiting regime where the particle
diameter $\d$ becomes small and the number of particles is of the order $\d^{-3}$.
We then show that the resulting flow map, restricted to points on the unit circle, 
converges to the coalescing Brownian flow.

\section{Some illustrations}
\label{illus}

In the Introduction we described the construction of a sequence of iterated conformal maps that can be
interpreted as a model of the aggregation of diffusive particles `coming in from infinity'. 
We chose the basic particle $P_1$ to be a closed disc of diameter $\d>0$. However our analysis 
turns out to apply to a wide range of basic particle shapes, as we shall see in Section \ref{model}.
Recall that $D_0$ denotes the complement of the closed unit disc $K_0$ in $\C\cup\{\infty\}$ and 
$D_1$ denotes the complement of $K_1$ in $\C\cup\{\infty\}$, where $K_1=K_0\cup P_1$. Write
$F_1$ for the conformal isomorphism $D_0\to D_1$ fixing $\infty$, with $F_1'(\infty)>0$ and write
$G_1$ for the inverse isomorphism $D_1\to D_0$. Denote the upper half-plane by $H$.

We now discuss briefly the form of $G_1$ for two particular choices of $P_1$.
The first of these, corresponding to the slit $[1,1+\d]$, was used to generate some
realizations of the cluster for various values of $\d$, which are presented in Figure \ref{DLAfig}.

\begin{figure}[p]
  \vspace{-20pt}
    \subfigure[The cluster after a few arrivals with $\d=1$.]
      {\label{DLAfew}
      \epsfig{file=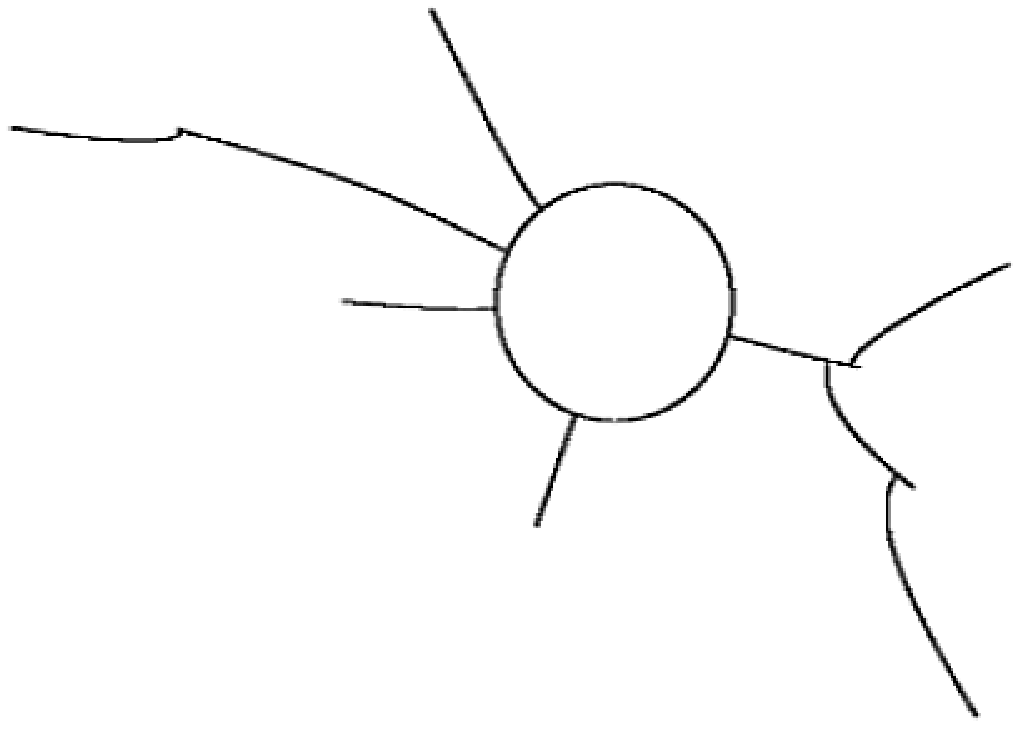,width=5.8cm}}
    \hfill
    \subfigure[The cluster after 100 arrivals with $\d=1$.]
      {\label{DLAmany}
      \epsfig{file=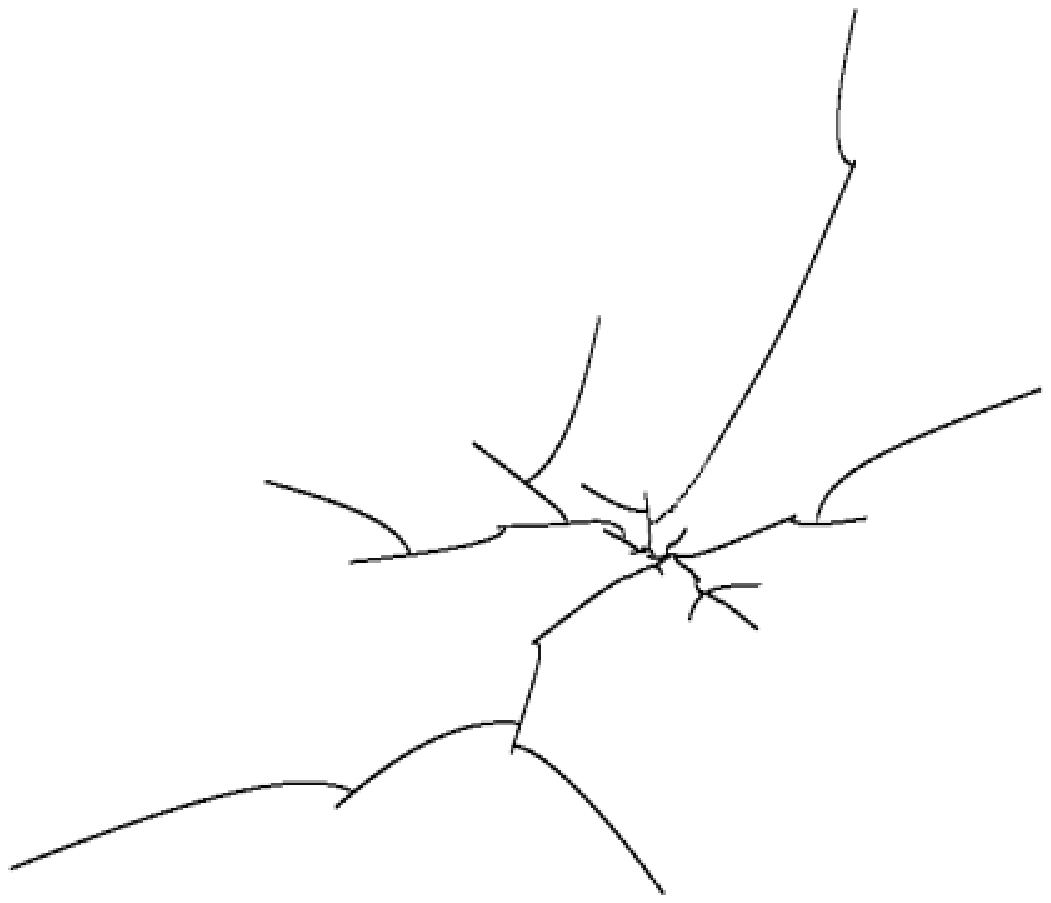,width=5.8cm}}
    \subfigure[The cluster after 800 arrivals with $\d=0.1$.]
      {\label{slit10}
      \epsfig{file=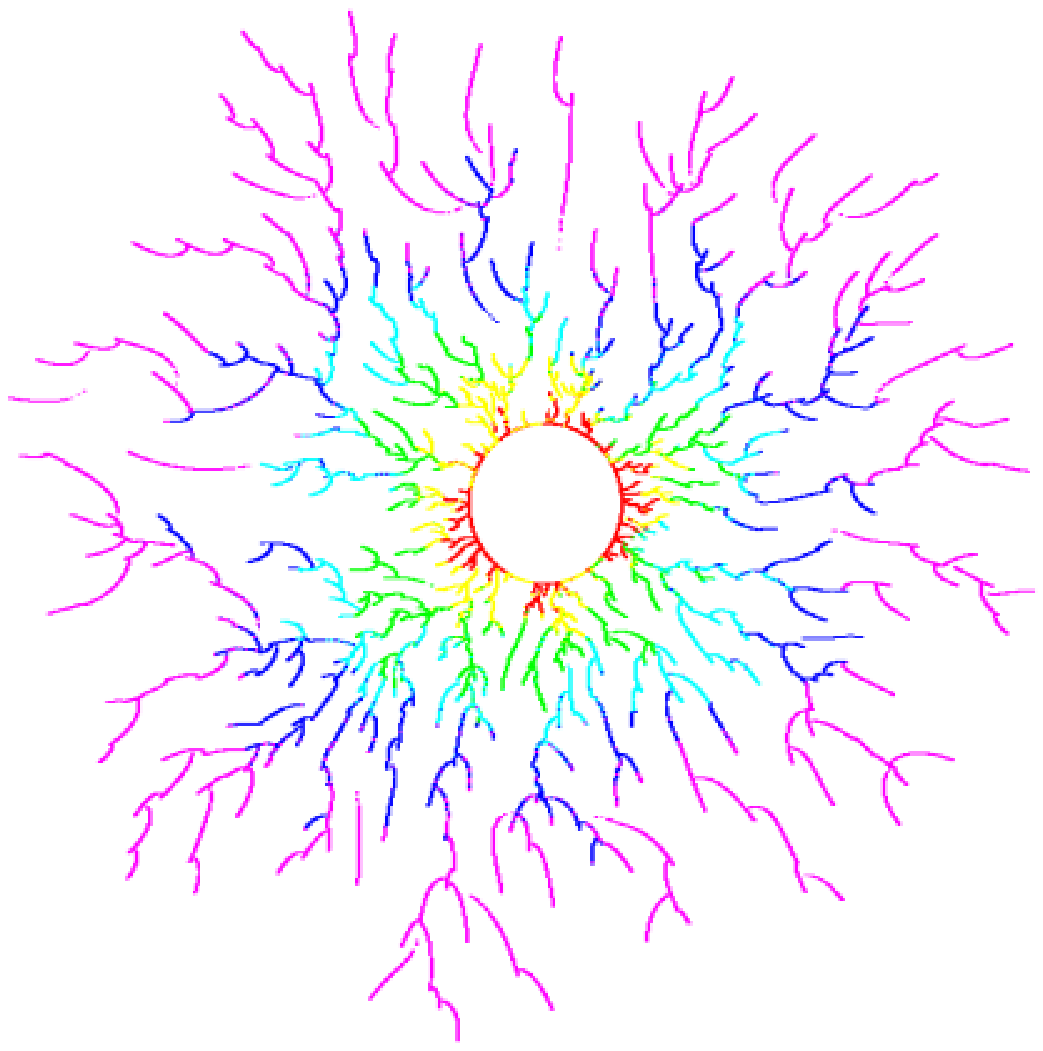,width=5.8cm}}
    \hfill
    \subfigure[The cluster after 5000 arrivals with $\d=0.04$.]
      {\label{slit25}
      \epsfig{file=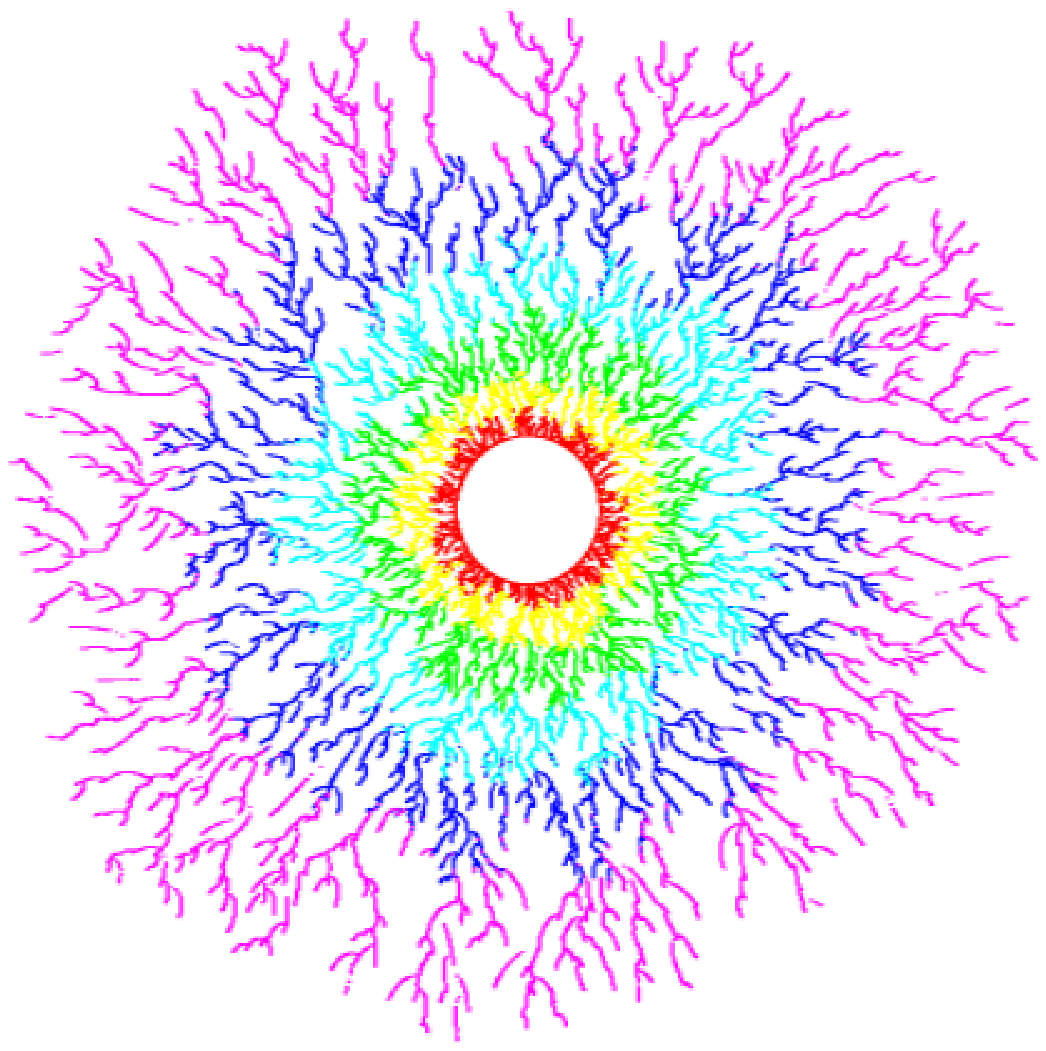,width=5.8cm}}
    \subfigure[The cluster after 20000 arrivals with $\d=0.02$.]
      {\label{slit50}
      \epsfig{file=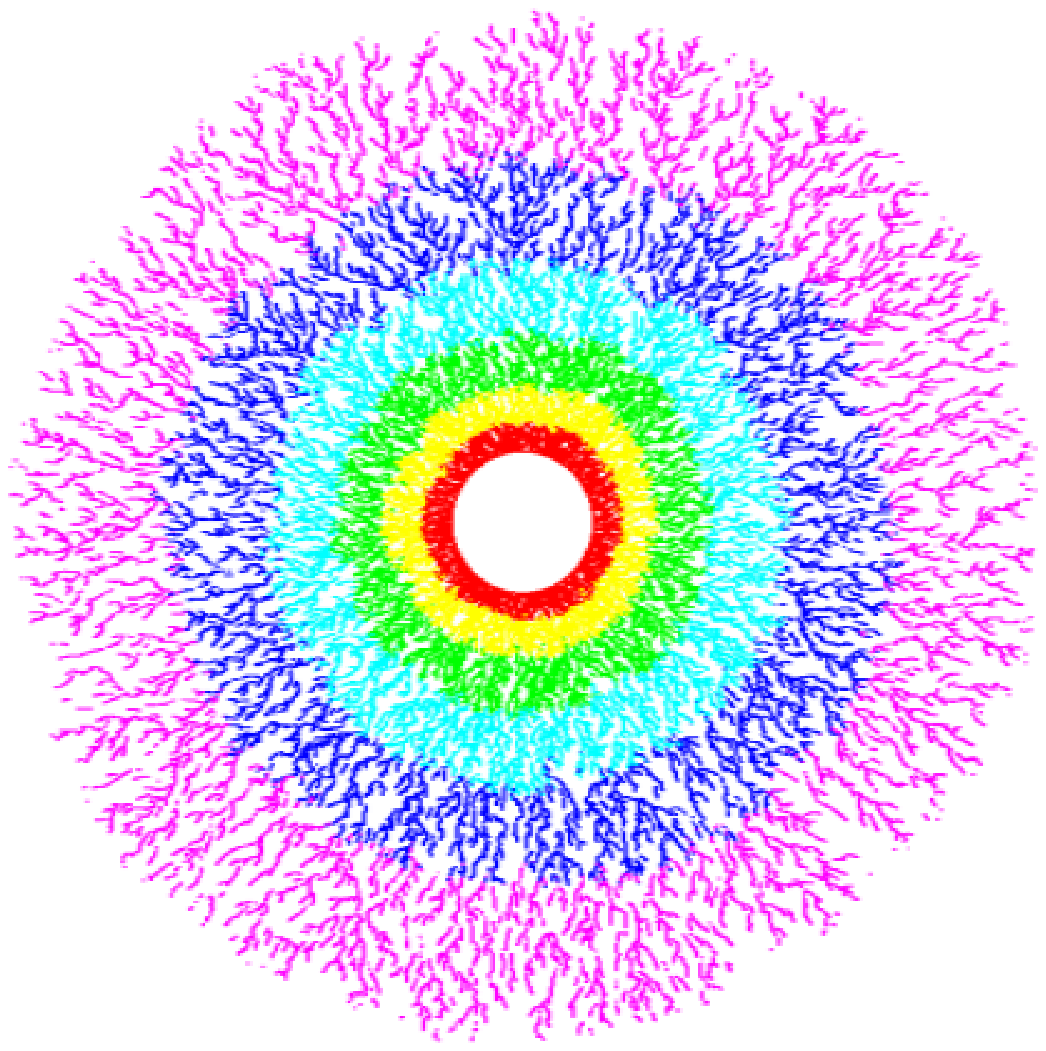,width=5.8cm}}
    \hfill
    \subfigure[The stochastic flow $(X_{t0})_{t \in [0,1]}$ with
    $\d=0.02$.]
      {\label{flow50}
      \epsfig{file=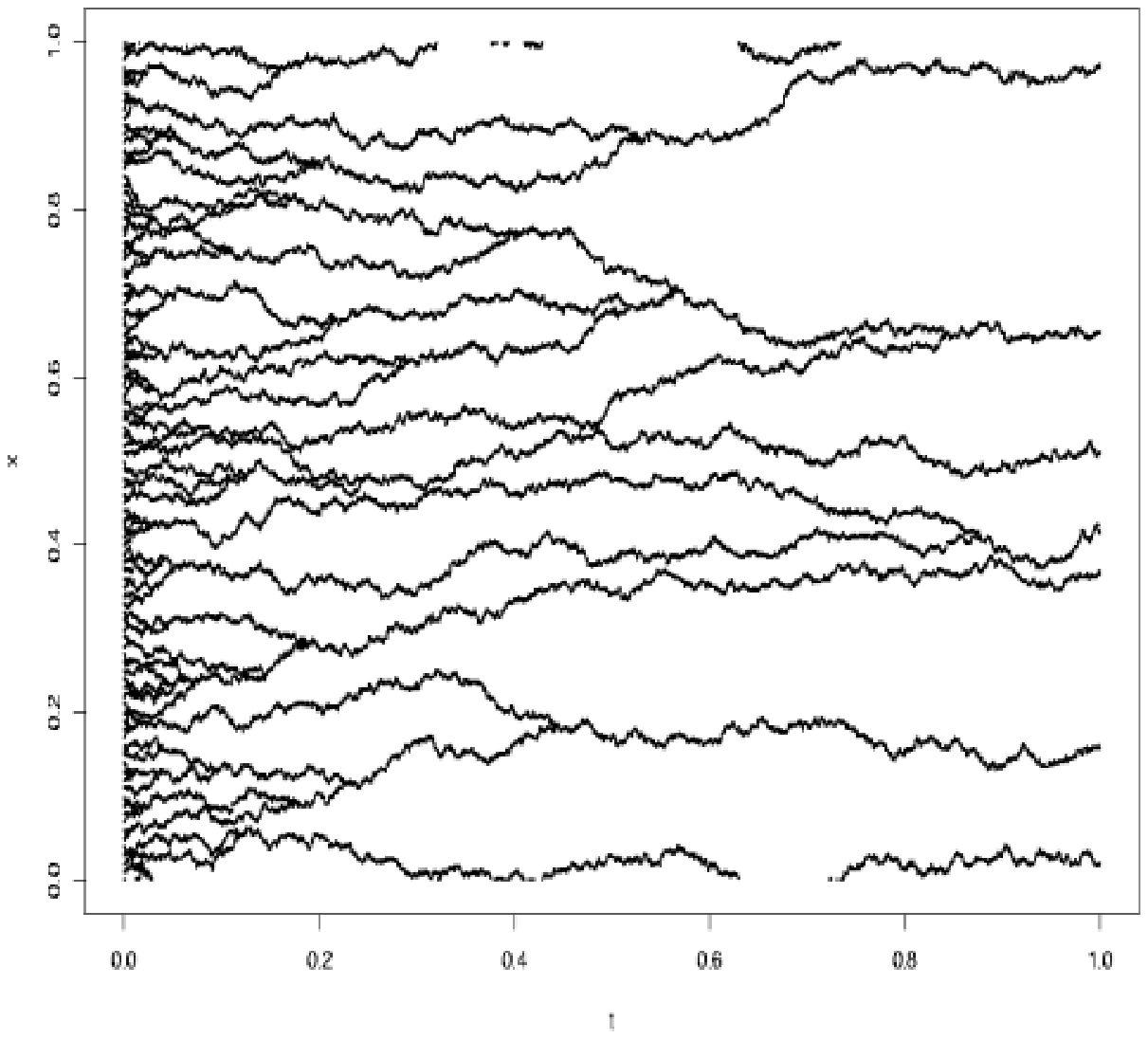,width=5.8cm}}
  \caption{\textsl{The slit model case of simplified Hastings-Levitov DLA}}
  \label{DLAfig}
\end{figure}

For the slit, we can obtain the map $G_1$ by composing the sequence of maps
$$
D_1\to H\sm (0,i\sqrt{t}]\to H\to H\to D_0,
$$
given by
$$
g_1(z)=i\frac{z-1}{z+1},\q
g_2(z)=\sqrt{z^2+t},\q
g_3(z)=\l z,\q
g_4(z)=\frac{i+z}{i-z}.
$$
Here, we choose $t=\d^2/(2+\d)^2$ so that $g_2\circ g_1(1+\d)=g_2(i\d/(2+\d))=0$ 
and choose $\l=1/\sqrt{1-t}$ so that
$g_3\circ g_2\circ g_1(\infty)=g_3\circ g_2(i)=\l i\sqrt{1-t}=i$. Then $G_1(1+\d)=1$
and $G_1(\infty)=\infty$, as required.
The maps extend to the boundaries of their domains in obvious ways. In particular the boundary
of the slit maps under $G_1$ to a boundary arc $[-\th_\d,\th_\d]$ of the unit circle. 
Note that $G_1(1)=g_4(\pm\l\sqrt{t})$ and $g_4'(0)=-2i$, from which we see that 
\begin{equation}\label{DTHN}
\frac{\th_\d}\d\to1\q\text{as}\q \d\to0. 
\end{equation}

We observe in Figure \ref{DLAmany}, when $\d=1$, that incoming particles are markedly distorted 
and, in particular, that particles arriving later tend to be larger. This effect is diminished
when we examine smaller values of $\d$. In Figure \ref{slit50}, the cluster is a rough ball, but with
some sort of internal structure. The colours label arrivals in different epochs, showing that growth
over time is uniform and there is a close relationship between time of arrival and the distance
from the origin at which the particle sticks. Figure \ref{flow50} focuses on the motion of points
on the original circular boundary, which is the aspect of these simulations examined theoretically
in this paper. The coalescing motion displayed agrees with our theoretical predictions. The size of
the gaps between the flow lines gives the amount of harmonic measure carried on fingers of the
cluster based between the chosen points on the unit circle.

\def\keep{  $$
  F_1^{-1}(z) = \left ( 1 - \left ( \frac{1}{1+2N}\right )^2 \right )
  \frac{z+1}{2z} \left ( z + 1 + \sqrt{z^2 + 1 - 2z \frac{ 1 +\left (
  \frac{1}{1+2N}\right )^2}{1 - \left ( \frac{1}{1+2N}\right )^2}
  }\right ) - 1.
  $$
  This map can be extended continuously to the boundary of the unit disc by setting $F_1^{-1}(e^{2 \pi i x}) = \exp (2 \pi f(x))$, for $x \in (0, 1)$, where
  $$
  f(x) =  \begin{cases}
          \pi^{-1} \tan^{-1} \sqrt{\frac{\tan^2 \pi x + (1 +
          2N)^{-2}}{1 - (1 + 2N)^{-2}}} & \quad x \in (0,
         \frac{1}{2}) \\
         - \pi^{-1} \tan^{-1} \sqrt{\frac{\tan^2 \pi x + (1 +
          2N)^{-2}}{1 - (1 + 2N)^{-2}}} & \quad x \in (\frac{1}{2},1).
         \end{cases}
  $$
}

In the case where $P_1$ is the lune $L=\{z\in\C:|z-1|\le\d,|z|>1\}$, the map $G_1$
is given by
$$
G_1(z)=\frac{(e^{i\theta}z-1)^a-(e^{-i\theta}z-1)^a}{(z-e^{-i\theta})^a-(z-e^{i\theta})^a},
$$
where
$$
\theta=2\tan^{-1}\d\sqrt{1-\d^2/4},
\q a=\frac{\pi}{\pi-\cos^{-1}(\d/2)}.
$$
We shall use in Section \ref{model} the following estimate of the 
logarithmic capacity of $K_1$ for this case. As $\d\to0$, 
\begin{equation}\label{LCE}
\cp(K_1)=-\log\lim_{z\rightarrow\infty}\frac{G_1(z)}z
                       =\log\frac{a\sin\theta}{\sin(a\theta)}=2\d^2+o(\d^2).
\end{equation}

\section{A class of L\' evy flows on the circle}\label{LFC}
We introduce a class of random flows on the circle, whose distributions
are invariant under rotations of the circle, and, under which, each point
on the circle performs a L\' evy process of mean $0$ with no diffusive part.
The flow maps are in general not continuous on the circle, but have a
non-crossing property. In a certain asymptotic regime, the
motion of the flow from a countable family of starting points
is shown to converge weakly to a family of
coalescing Brownian motions.

We specify a particular flow by the choice of a non-decreasing,
right-continuous function $f^+:\R\to\R$ with the following {\em degree}
$1$ property\footnote{These functions can be considered as liftings of maps from the circle
$\R/\Z$ to itself having an order-preserving property.
In practice, the circle map will be a perturbation of the identity map and
our basic map $f^+$ will be the unique lifting which is close to the identity map on $\R$.}
\begin{equation}
\label{circdef}
f^+(x+n) = f^+(x) + n, \quad x\in\R,\q n \in \mathbb{Z}.
\end{equation}
Denote the set of such functions by $\cR$ and write $\cL$ for the
analogous set of left-continuous functions.
Each $f^+\in\cR$ has a left-continuous modification $f^-\in\cL$, given by $f^-(x)=f(x-)$.
Write $\cD$ for the set of all pairs $f=\{f^-,f^+\}$.
When $f^+$ is continuous, we shall also write $f=f^+$ and, generally, we write $f$ in place of $f^\pm$
in expressions where the choice of left or right-continuous modification makes no difference to the value.
For $f^+\in\cR$ and $z=\tilde z+\Z\in\R/\Z$, define $(\t_z f)^+\in\cR$ by
$$
(\t_z f)^+(x)=\tilde z+f^+(x-\tilde z),\q x\in\R.
$$
The degree $1$ property ensures that there is no dependence on the choice of representative $\tilde z$.
The sets $\cR$ and $\cL$ are closed under composition, but $\cD$ is not.
In fact, if $f_1,f_2\in\cD$, then $f_2^-\circ f_1^-$ is the left-continuous modification of $f_2^+\circ f_1^+$
if and only if $f_1$ sends no interval of positive length to a point of discontinuity of $f_2$.
We say in this case that $f_2\circ f_1\in\cD$, denoting by $f_2\circ f_1$ the pair $\{f_2^-\circ f_1^-,f_2^+\circ f_1^+\}$.
Define $\id(x)=x$ and set $\cD^*=\cD\sm\{\id\}$.
We assume throughout that our basic map $f\in\cD^*$.
Write $\tilde f^\pm$ for the periodic functions $\tilde f^\pm(x)=f^\pm(x)-x$.
Define constants  $\rho=\rho(f)>0$ and $\b=\b(f)\in\R$
by
\begin{equation}
\label{lambdadef}
\rho \int_0^1\tilde{f}(x)^2 dx = 1, \q \beta =\rho\int_0^1\tilde{f}(x) dx.
\end{equation}

We now define a probability space $(\O,\cF,\PP)$. Take $\O$ to be the set
of those integer-valued Borel measures on $\R\times(\R/\Z)$
which are finite on bounded sets, and also have the property that
$\omega(\{t\} \times(\R/\Z)) \in \{0,1\}$ for all $t \in \mathbb{R}$.
We interpret the first factor $\R$ of the underlying space as time.
Write $\cF^o$ for the $\s$-algebra on $\O$ generated by evaluations on Borel sets.
For intervals $I\sse\R$, write $\cF_I^o$
for the $\s$-algebra on $\O$ generated by evaluations on Borel subsets of $I\times(\R/\Z)$.
There is a unique probability measure $\PP = \PP^{\rho}$ on
$(\O,\cF^o)$ which makes the coordinate map
$\mu(\o,dt,dz)=\o(dt,dz)$ into a Poisson random measure on
$\R\times(\R/\Z)$ with intensity $\nu(dt,dz)=\rho\, dtdz$.
Here $dz$ denotes Lebesgue measure on $\R/\Z$: we shall use,
without further comment, in defining integrals,
the obvious identification of $\R/\Z$ and $[0,1)$. 
Write $\cF$ for the completion of $\cF^o$ with respect to $\PP$,
extending $\PP$ to $\cF$ as usual.

We first construct the flow in the case where $\b=0$.
Given $\omega \in \Omega$, for each $t \in \mathbb{R}$ define $F_t=\{F_t^-,F_t^+\}\in \cD$ by
$$
F_t^\pm = \begin{cases}
          (\t_z f)^\pm, & \text{if $\o$ has an atom at $(t,z)$}, \\
          \id,          & \text{otherwise}.
         \end{cases}
$$
For each finite interval $I\sse\R$, define $X_I^+\in\cR$ and $X_I^-\in\cL$ by
$$
X_I^\pm=F_{T_n}^\pm\circ\dots\circ F_{T_1}^\pm,
$$
where $T_1<\dots<T_n$ are the times of the atoms of $\o$ in $I\times(\R/\Z)$.
We take $X_I^\pm=\id$ if there are no such atoms.

In the case where $\b\not=0$, we replace $\o$ in the preceding construction
by $\o^\b$, given by
$$
\o^\b(dt,dz)=\o(dt,d(z+\b t)),
$$
to obtain $X^{\b,\pm}_I$, and then set
$$
X_I^\pm(x)=X_I^{\b,\pm}(x+\b s)-\b t,
$$
where $s=\inf I$ and $t=\sup I$.

Since $f$ can have at most countably many points of discontinuity and intervals of
constancy, and since, under $\PP$, the positions of the atoms of $\o$ are
distributed uniformly on $\R/\Z$, we have, almost surely, $X_I=\{X_I^-,X_I^+\}\in\cD$ for all
intervals $I$.
Write $I=I_1\oplus I_2$ if $I_1, I_2$ are disjoint intervals with $\sup I_1=\inf I_2$ and $I=I_1\cup I_2$.
Note that $(X_I:I\sse\R)$ has the following properties: 
\begin{align}\label{LFM}
&X^+_I(x)\text{ and $X^-_I(x)$ are random variables for all finite intervals $I$ and all $x\in\R$},\\
\label{LWF}
&X_I^+=X_{I_2}^+\circ X_{I_1}^+\text{ and $X_I^-=X_{I_2}^-\circ X_{I_1}^-$ whenever $I=I_1\oplus I_2$},\\
\label{LPC}
&X^+_{(s,t)}(x)=X^-_{(s,t)}(x)=x-\b(t-s)\text{ for all $x\in\R$, eventually as $s\ua t$ or $t\da s$}.
\end{align}

Fix $e=(s,x)\in\R^2$, and define two processes $X^{e,-}_t$ and $X^{e,+}_t$,
both starting from $e$, by setting $X_t^{e,\pm}=X^\pm_{(s,t]}(x)$ for $t\ge s$.
Then $X^{e,-}$ and $X^{e,+}$ are both piecewise continuous, cadlag, and satisfy the integral equations
$$
X_t^{e,\pm}=x+\int_{(s,t]\times[0,1)}\tilde f^\pm(X^{e,\pm}_{r-}-z)(\mu-\nu)(dr,dz),
\q t\ge s.
$$
Under $\PP$, we have $X^{e,-}_t=X^{e,+}_t$ for all $t\ge s$, almost surely.
We shall therefore drop the $\pm$ from now on and write simply $X^e$.
Write $\mu^f_e$ for the distribution of $X^e$ under $\PP$
on the Skorokhod space $D_e=D_x([s,\infty),\R)$ of cadlag paths starting from $x$ at time $s$.
Write $\mu_e$ for the distribution on $D_e$ of a standard Brownian motion
starting from $e$.

In the context of planar aggregation $f$ describes the action, on the boundary of the unit disc, of the 
conformal map corresponding to the arrival of a particle. This action is descibed precisely in 
Section \ref{model}. Where particles are symmetric around the axis through their attachment point, $\b=0$. 
Asymmetric particles can give rise to non-zero values of $\b$ and in these cases a drift is induced for which we have compensated.

\begin{proposition}\label{LBM}
We have $\mu^f_e\to\mu_e$ weakly on $D_e$, uniformly in $f\in\cD^*$ as $\rho(f)\to\infty$.
\end{proposition}
\begin{proof}
We drop the superscript $e$ within the proof to lighten the notation.
Under $\PP$, the process $(X_t)_{t\ge s}$ is a martingale. In fact, it is a
L\' evy process, with characteristic exponent $\chi$ given by
$$
\chi(\th)=\rho\int_0^1\left\{e^{i\th\tilde f(z)}-1-i\th\tilde f(z)\right\}dz,\q\th\in\R.
$$
In particular, we have
$\E(|X_{t_1}-X_{t_2}|^2)=|t_1-t_2|$ for all $t_1,t_2\ge s$.
A standard criterion (see for example \cite[page 143]{B} or \cite[page 355]{MR1943877}) allows us
to deduce that the family of laws $(\mu^f_e:f\in\cD^*)$ is tight in $D_e$.
We note that the process $(X_t^2-t)_{t\ge s}$ is also a martingale,
and that the jumps of $(X_t)_{t\ge s}$ are bounded in absolute
value by $\|\tilde f\|=\sup_{x\in\R}|\tilde f(x)|$.
Let $\mu$ be any weak limit law for the limit $\rho\to\infty$.
Write $(Z_t)_{t\ge s}$ for the coordinate process on $D_e$.
Under $\mu$, by standard arguments,
both $(Z_t)_{t\ge s}$ and $(Z_t^2-t)_{t\ge s}$
are local martingales in the natural filtration of $(Z_t)_{t\ge s}$.
By using the fact that $f$ is non-decreasing, we can obtain for all $\rho\ge1$ the estimate
\begin{equation}\label{FR3}
\|\tilde f\|\le2\rho^{-1/3}.
\end{equation}
Hence $\mu$ is supported on continuous paths and must therefore be $\mu_e$
by L\' evy's characterization of Brownian motion.
\end{proof}

We next fix a sequence $E=(e_k:k\in\N)$ in $\R^2$, where $e_k=(s_k,x_k)$ say,
and consider the sequence of processes $X^E=(X^k:k\in\N)$, where $X^k=X^{e_k}$.
Then $X^E$ is a random variable in the complete separable
metric space $D_E=\prod_{k=1}^\infty D_{e_k}$, where we define the metric $d_E$ on $D_E$ by
$$
d_E(\xi,\xi')=\sum_{k=1}^\infty 2^{-k}\{d(\xi^k,{\xi'}^k)\wedge 1\},
$$
and where $d$ denotes appropriate instances of the Skorokhod metric.
Write $\mu^f_E$ for the distribution of $X^E$ on $D_E$ under $\PP$.

We consider now a limit in which the basic map $f$ is an increasingly well-localized perturbation of
the identity, where we quantify this property in terms of the smallest constant
$\l=\l(f)\in(0,1]$ such that
$$
\rho\int_0^1|\tilde f(x+a)\tilde f(x)|dx\le\l,\q a\in[\l,1-\l].
$$
Denote by $(Z^k_t)_{t\ge s_k}$ the $k$th coordinate process on $D_E$,
given by $Z_t^k(\xi)=\xi^k_t$, and consider the
filtration $(\cZ_t)_{t\in\R}$ on $D_E$, where $\cZ_t$ is
the $\s$-algebra generated by $(Z^k_s:s_k<s\le t\vee s_k,k\in\N)$.
Write $C_E$ for the (measurable) subset of $D_E$ where each coordinate path is continuous.
Define also
$$
T^{jk}=\inf\{t\ge s_j\vee s_k:Z_t^j-Z_t^k\in\Z\}.
$$
We are thinking of the paths $(Z^k_t)_{t\ge s_k}$ as liftings of paths in the circle. Thus
$T^{jk}$ is the collision time of these circle-valued paths.
The following is a convenient reformulation of a result of Arratia \cite{A79}.
\begin{proposition}\label{ARRA}
There exists a unique Borel probability measure $\mu_E$
on $D_E$ under which, for all $j,k\in\N$, the processes
$(Z^k_t)_{t\ge s_k}$ and $(Z^j_tZ^k_t-(t-T^{jk})^+)_{t\ge s_j\vee s_k}$ are
both continuous local martingales in the filtration $(\cZ_t)_{t\in\R}$.
\end{proposition}
Of course $\mu_E$ is supported on $C_E$ and may naturally be considered as a measure
defined there.
We sketch a proof. For existence, one can take independent
Brownian motions from each of the given time-space starting points
and then impose a rule of coagulation on collision, deleting the path of lower index.
The law of the resulting
process has the desired properties. On the other hand, given a probability
measure such as described in the proposition, on some larger
probability space, one can use a supply of independent Brownian motions
to resurrect the paths deleted at each collision. Then L\' evy's
characterization can be used to see that one has recovered
the set-up used for existence. This gives uniqueness.

\begin{proposition}\label{LAF}
We have $\mu^f_E\to\mu_E$ weakly on $D_E$,
uniformly in $f\in\cD^*$, as $\rho(f)\to\infty$ and $\l(f)\to0$.
\end{proposition}
\begin{proof}
The families of marginal laws $(\mu^f_{e_k}:f\in\cD^*)$ are all tight, as shown
in Proposition \ref{LBM}.
Hence the family of laws $(\mu^f_E:f\in\cD^*)$ is also tight.
Let $\mu$ be any weak
limit law for the limits $\rho\to\infty$ and $\l\to0$. Under $\PP$, for $j,k$
distinct, the process
$$
X^j_tX^k_t-\int_{s_j\vee s_k}^tb(X^j_s,X^k_s)ds, \q t\ge s_j\vee s_k,
$$
is a martingale, where
$$
b(x,x')=\rho\int_0^1\tilde f(x-z)\tilde f(x'-z)dz.
$$
We have $|b(x,x')|\le\l$ whenever $\l\le|x-x'|\le1-\l$. Hence, by standard
arguments, under $\mu$, the process $(Z^j_tZ^k_t:s_j\vee s_k\le t<T^{jk})$ is a local
martingale. We know from the proof of Proposition \ref{LBM} that, under $\mu$, the processes
$(Z^j_t:t\ge s_j)$, $((Z^j_t)^2-t:t\ge s_j)$ and $(Z^k_t:t\ge s_k)$ are continuous local martingales.
But $\mu$ inherits from the laws $\mu^f_E$ the property
that, almost surely, for all $n\in\Z$, the process $(Z^j_t-Z^k_t+n:t\ge s_j\vee s_k)$ does not change sign.
Hence, by an optional stopping argument, $Z^j_t-Z^k_t$ is constant for $t\ge T^{jk}$.
It follows that $(Z^j_tZ^k_t-(t-T^{jk})^+)_{t\ge s_j\vee s_k}$ is a continuous local martingale.
Hence $\mu=\mu_E$, by Proposition \ref{ARRA}.
\end{proof}

We note that the characterizing property of $\mu_E$ is invariant under a permutation of the
sequence $(e_k:k\in\N)$, as is the topology of the metric space $D_E$. 
Hence we can define for any countable set $E\sse\R^2$, a unique Borel probability measure 
on $\mu_E$ on $D_E=\prod_{e\in E}D_e$, having the given property. Then, by Proposition \ref{LAF}, the 
distribution $\mu_E^f$ of $X^E=(X^e:e\in E)$ on $D_E$ converges weakly to $\mu_E$ as $\rho(f)\to\infty$
and $\l(f)\to0$.

\section{A new state-space for the coalescing Brownian flow}
\label{CBF}
The preceding statement is unsatisfactory in that it expresses convergence
of our L\' evy flows only for a given countable set of time-space starting points. To remedy this, we
must first formulate a suitable limit object. 
This object is known as Arratia's flow, or the Brownian web,
and has been studied in some depth. However, we have found it convenient to introduce a new
state-space of flows, which we now describe.

We begin by defining a metric on $\cD$.
Let $\cS$ denote the set of all periodic contractions on $\mathbb{R}$ having period 1.
Each $f\in\cD$ can be identified with some $f^\times\in\cS$ by drawing new
axes at an angle $\pi/4$ with the old, and scaling
appropriately. See Figure \ref{Phimap}.
\begin{figure}[ht]
  \centering
  \epsfig{file=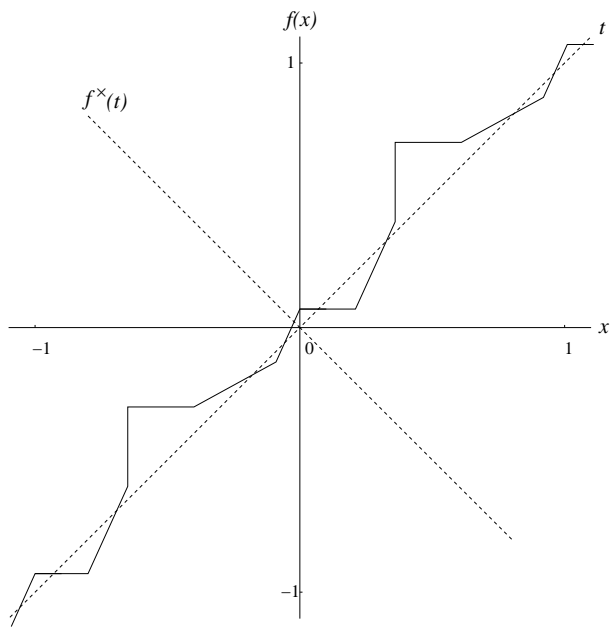, width=8cm}
  \caption{\textsl{The map $f^\times$ obtained from $f$ by rotating
  the axes by $\frac{\pi}{4}$.}}
\label{Phimap}
\end{figure}
More formally, since $x+f^+(x)$ is strictly increasing in $x$, there is for each $t\in\R$
a unique $x\in\R$ such that
$$
\frac{x+f^-(x)}2\le t\le\frac{x+f^+(x)}2.
$$
Define $f^\times(t)=t-x$. Note that $\id^\times=0$.
Then {\em the map $f\mapsto f^\times:\cD\to\cS$ is a bijection},
so we can define a metric $d_\cD$ on $\cD$ by
$$
d_\cD(f,g)=\|f^\times-g^\times\|=\sup_{t\in[0,1)}|f^\times(t)-g^\times(t)|.
$$
A proof of the italicized assertion is given in the Appendix. The same is true for some 
further technical assertions which will be made below, written also in italics. 
The metric space $(\cS,\|\dots\|)$ is complete and locally compact, 
so the same is true for $(\cD,d_\cD)$.
An alternative characterization\footnote{Thus, $d_\cD$ is a close relative
of the L\' evy metric sometimes used on the set of distribution functions for real random
variables. The relationships of such a metric to the operations of composition and inversion 
in $\cD$, which are significant for us, do not appear to have been studied.}
of the metric $d_\cD$ is as follows: {\em for $f,g\in\cD$ and $\ve>0$, we have}
$$
d_\cD(f,g)\le\ve\iff f^-(x-\ve)-\ve\le g^-(x)\le g^+(x) \le f^+(x+\ve)+\ve
\text{ {\em for all }}x\in\R.
$$
We deduce that, for $f,g\in\cD$,
$$
d_\cD(f,g)\le\|f-g\|,\q 2d_\cD(f,\id)=\|f-\id\|,
$$
and
$$
d_\cD(f,g\circ f)\le\|g-\id\|\text{ when $g\circ f\in\cD$},\q
d_\cD(f,f\circ g)\le\|g-\id\|\text{ when $f\circ g\in\cD$}.
$$
Moreover, {\em for any sequence $(f_n:n\in\N)$ in $\cD$,
$$
f_n\to f\iff f_n(x)\to f(x)\text{ at every point $x$ where $f$ is continuous}.
$$}
Here and below, we write $f_n\to f$ to mean convergence in the metric $d_\cD$.

We now define our space of flows. We call them weak flows to emphasise that the
usual flow property may fail at points of spatial discontinuity.
Consider $\phi=(\phi_{ts}:s,t\in\R,s<t)$, with $\phi_{ts}\in\cD$ for all $s,t$.
Say that $\phi$ is a {\em weak flow} if
$$
\phi_{ut}^-\circ\phi_{ts}^-\le\phi_{us}^-\le\phi_{us}^+\le\phi_{ut}^+\circ\phi_{ts}^+,
\q s<t<u.
$$
Say that $\phi$ is {\em continuous} if, for all $t\in\R$,
$$
\phi_{ts}\to\id\text{ as $s\ua t$},\q\phi_{ut}\to\id\text{ as $u\da t$}.
$$
Write $C^\circ(\R,\cD)$ for the set of all continuous weak flows\footnote{\label{CDZ}In Section \ref{SLAM} we shall work with
a modified space of flows $C^\circ((0,\infty),\cD_0)$. Here we restrict to intervals $I\sse(0,\infty)$, and we take
$\phi_I\in\cD_0$, where $\cD_0$ is the set of circle maps whose liftings are in $\cD$. Given the continuity of the map (\ref{PHIC})
and the requirement $\phi_{ss}=\id$, there is an obvious identification of this space with $C^\circ((0,\infty),\cD)$}.
It will be convenient sometimes to extend a continuous weak flow $\phi$ to the diagonal, which we do
by setting $\phi_{ss}=\id$ for all $s\in\R$. 
Then, {\em for any $\phi\in C^\circ(\R,\cD)$, the map
\begin{equation}\label{PHIC}
(s,t)\mapsto\phi_{ts}:\{(s,t):s\le t\}\to\cD
\end{equation}
is continuous.}

Define, for $\phi,\psi\in C^\circ(\R,\cD)$,
$$
d_C(\phi,\psi)=\sum_{n=1}^\infty2^{-n}\{d_C^{(n)}(\phi,\psi)\wedge1\},
$$
where
$$
d_C^{(n)}(\phi,\psi)=\sup_{s,t\in(-n,n),s<t}d_{\cD}(\phi_{ts},\psi_{ts}).
$$
Then $d_C$ is a metric on $C^\circ(\R,\cD)$, {\em under which this space is
complete and separable}.
{\em For this metric, for all $s,t\in\R$ with $s<t$, and all $x\in\R$, the {\rm evaluation map}
$$
\phi\mapsto\phi_{ts}^+(x):C^\circ(\R,\cD)\to\R
$$ 
is continuous.
Moreover the Borel $\s$-algebra on $C^\circ(\R,\cD)$ is generated by
the set of all such evaluation maps with $s,t$ and $x$ rational.}
                                                                                                                                                                            
{\em For $e=(s,x)\in\R^2$ and $\phi\in C^\circ(\R,\cD)$, the map
$$
t\mapsto\phi_{ts}^+(x):[s,\infty)\to\R
$$
is continuous.} Hence we can define a measurable map
$Z^{e,+}:C^\circ(\R,\cD)\to C_e=C_x([s,\infty),\R)$ by
$$
Z^{e,+}(\phi)=(\phi_{ts}^+(x):t\ge s).
$$
Given a countable set $E\sse\R^2$, we then define a measurable map 
$Z^{E,+}:C^\circ(\R,\cD)\to C_E=\prod_{e\in E}C_e$ by
$$
Z^{E,+}(\phi)^e=Z^{e,+}(\phi).
$$
Set $C_E^{\circ,+}=\{Z^{E,+}(\phi):\phi\in C^\circ(\R,\cD)\}$.
Define similarly $Z^{e,-}$, $Z^{E,-}$ and $C_E^{\circ,-}$ and set $C_E^\circ=C_E^{\circ,+}\cap C_E^{\circ,-}$.

The following result translates into the language of continuous weak flows a result of T\'oth and Werner \cite[Theorem 2.1]{TW}, 
which itself was a variant of a result of Arratia \cite{A79}.
We shall give a complete proof, in part because we need most components of the proof also for our main convergence result,
and in part because our framework leads to some simplifications, for example in the probabilistic underpinnings
contained in Proposition \ref{MPI}.
The formulation in terms of continuous weak flows has advantages, in 
leading to a unique object, with a natural time-reversal invariance, and for the derivation of weak limits.
Recall that, for a countable set $F\sse\R^2$, we denote by $\mu_F$ the law on $C_F$ of a family of coalescing Brownian motions starting
from $F$, as discussed at the end of the preceding section.

\begin{theorem}\label{UBPR}
There exists a unique Borel probability measure $\mu_A$ on $C^\circ(\R,\cD)$
such that, for any finite set $F\sse\R^2$, we have
\begin{equation}\label{MUW}
\mu_A\circ (Z^{F,+})^{-1}=\mu_F.
\end{equation}
Moreover, for all $e\in\R^2$, we have, $\mu_A$-almost surely, $Z^{e,+}=Z^{e,-}$. 
\end{theorem}
\begin{proof}
Fix a countable subset $E$ of $\R^2$ containing $\Q^2$.
Consider $\phi,\psi\in C^\circ(\R,\cD)$ and suppose that $Z^{E,+}(\phi)=Z^{E,+}(\psi)$. Then $\phi_{ts}^+(x)=\psi_{ts}^+(x)$ for all
$s,x\in\Q$ and all $t>s$. This extends to all $x\in\R$ because $\phi_{ts}^+$ and $\psi_{ts}^+$ are both right continuous,
and then to all $s\in\R$ using (\ref{PHIC}), so $\phi=\psi$. 
Hence the map $Z^{E,+}:C^\circ(\R,\cD)\to C_E^{\circ,+}$ is a bijection.
Write $\Phi^{E,+}$ for the inverse bijection $C^{\circ,+}_E\to C^\circ(\R,\cD)$. 
Similarly, write $\Phi^{E,-}$ for the inverse of the bijection $Z^{E,-}:C^\circ(\R,\cD)\to C_E^{\circ,-}$.
Then {\em $C^{\circ,+}_E$ and $\Phi^{E,+}$ are
measurable, we have $\mu_E(C^\circ_E)=1$, and $\Phi^{E,+}=\Phi^{E,-}$ on $C^\circ_E$.} 
So we can define a Borel probability measure $\mu_A$ on $C^\circ(\R,\cD)$ by
$$
\mu_A=\mu_E\circ(\Phi^{E,+})^{-1}.
$$
Then
$$
\mu_E=\mu_A\circ(Z^{E,+})^{-1}.
$$
For $F\sse E$, we have $Z^{F,+}=\pi_{E,F}\circ Z^{E,+}$,
and $\mu_F=\mu_E\circ\pi_{E,F}^{-1}$,
where $\pi_{E,F}:C_E\to C_F$ is the obvious projection.
So
\begin{equation*}
\mu_F=\mu_A\circ(Z^{E,+})^{-1}\circ\pi_{E,F}^{-1}=\mu_A\circ(Z^{F,+})^{-1}.
\end{equation*}
On the other hand, for $E\sse E'$, we have $Z^E=\pi_{E',E}\circ Z^{E'}$,
and $\mu_E=\mu_{E'}\circ\pi_{E',E}^{-1}$.
Then $\pi_{E',E}$ restricts to a bijection 
$C^{\circ,+}_{E'}\to C^{\circ,+}_E$, so we obtain that
$$
\mu_A=\mu_{E'}\circ\pi_{E',E}^{-1}\circ\Phi_E^{-1}=\mu_{E'}\circ\Phi_{E'}^{-1}.
$$
This shows that $\mu_A$ does not depend on $E$, and so (\ref{MUW})
holds for all finite sets $F$.
This property then characterizes $\mu_A$ by a standard $\pi$-system argument.
Finally, $\mu_A$-almost surely, $Z^{E,+}\in C_E^\circ$, so $Z^{E,+}=Z^{E,-}$,
so $Z^{e,+}=Z^{e,-}$ for all $e\in E$.
\end{proof}

We call any $C^\circ(\R,\cD)$-valued random variable with law $\mu_A$ a {\em coalescing Brownian flow}.
The relationship of this space and measure to the Brownian web of Fontes et al. is explored in the Appendix.

\section{A Skorokhod-type space of non-decreasing flows on the circle}
\label{SKOR}
Since our L\'evy flows are not continuous in time, it will be necessary to
introduce a larger flow space to accommodate them.
Consider now $\phi=(\phi_I:I\sse\R)$, where $\phi_I\in\cD$ and $I$ ranges
over all non-empty finite intervals. 
Recall that we write $I=I_1\oplus I_2$ if $I_1, I_2$ are disjoint intervals with $\sup I_1=\inf I_2$ and $I=I_1\cup I_2$. 
Say that $\phi$ is a {\em weak flow} if,
\begin{equation}\label{WF}
\phi_{I_2}^-\circ\phi_{I_1}^-\le\phi_I^-\le\phi_I^+\le\phi_{I_2}^+\circ\phi_{I_1}^+,
\q I=I_1\oplus I_2.
\end{equation}
Say that $\phi$ is {\em cadlag}\footnote{This definition is more symmetrical in time than
is usual for `cadlag': a more accurate acronym would be {\em laglad}}
if, for all $t\in\R$,
$$
\phi_{(s,t)}\to\id\q\text{as $s\uparrow t$},\q
\phi_{(t,u)}\to\id\q\text{as $u\downarrow t$}.
$$
Write $D^\circ(\R,\cD)$ for the set of cadlag weak flows.
It will be convenient to extend a cadlag weak flow $\phi$ to the empty interval by setting $\phi_\es=\id$.
Given a finite interval $I$ and a sequence of finite intervals $(I_n:n\in\N)$, write $I_n\to I$ if
$$
I=\bigcup_n\bigcap_{m\ge n}I_m=\bigcap_n\bigcup_{m\ge n}I_m.
$$
{\em For any $\phi\in D^\circ(\R,\cD)$, we have}
\begin{equation}\label{PHID}
\phi_{I_n}\to\phi_I\q\text{{\em as} $I_n\to I$}.
\end{equation}

Let $\phi$ be a cadlag weak flow and suppose that $\phi_{\{t\}}=\id$ for all $t\in\R$.
Then, using \eqref{WF}, we have $\phi_{(s,t)}=\phi_{(s,t]}=\phi_{[s,t)}=\phi_{[s,t]}$ for all $s<t$ and, denoting
all these functions by $\phi_{ts}$, the family $(\phi_{ts}:s,t\in\R,s<t)$ is a continuous weak
flow in the sense of the preceding section.

For $\phi,\psi\in D^\circ(\R,\cD)$ and $n\ge1$, define
$$
d_D^{(n)}(\phi, \psi) = \inf_\lambda\left\{
\gamma(\lambda) \vee \sup_{I\sse\R}\|\chi_n(I)\phi_I^\times-\chi_n(\l(I))\psi_{\lambda(I)}^\times\|\right\},
$$
where the infimum is taken over the set of increasing homeomorphisms
$\lambda$ of $\R$, where
$$
\gamma(\lambda) = \sup_{t\in\R}|\l(t)-t|\vee\sup_{s,t\in\R,s<t}\, \left|\log\left(\frac{\lambda(t) - \lambda(s)}{t-s}\right)\right|,
$$
and where $\chi_n$ is the cutoff function\footnote{As in the case of the standard Skorokhod topology, localization
in time sits awkwardly with the stretching of time introduced via the homeomorphisms $\l$. There is no fundamental
obstacle, just some messiness at the edges. We note that, when $I\cup\l(I)\sse[-n,n]$, we have
$$
\|\chi_n(I)\phi_I^\times-\chi_n(\l(I))\psi_{\lambda(I)}^\times\|=d_\cD(\phi_I,\psi_{\l(I)}).
$$
Also, for all intervals $I$, we have $|\chi_n(\l(I))-\chi_n(I)|\le\g(\l)$ and
$$
\|\chi_n(I)\phi_I^\times-\chi_n(\l(I))\psi_{\lambda(I)}^\times\|
\le\chi_n(I)d_\cD(\phi_I,\psi_{\l(I)})+|\chi_n(\l(I))-\chi_n(I)|\|\psi^\times_{\l(I)}\|.
$$}
given by
$$
\chi_n(I)=0\vee(n+1-R)\wedge1,\q R=\sup I\vee(-\inf I).
$$
Then define
\begin{equation}\label{SMET}
d_D(\phi, \psi)=\sum_{n=1}^\infty2^{-n}\{d_D^{(n)}(\phi,\psi)\wedge1\}.
\end{equation}
Then $d_D$ is a metric on $D^\circ(\R,\cD)$ {\em under which $D^\circ(\R,\cD)$ is
complete and separable.} 
Moreover {\em the metrics $d_C$ and $d_D$ generate the same topology on $C^\circ(\R,\cD)$}.
{\em For the metric $d_D$, for all finite intervals $I$ and all $x\in\R$, the {\rm evaluation map}
$$
\phi\mapsto\phi_I^+(x):D^\circ(\R,\cD)\to\R
$$
is Borel measurable.
Moreover the Borel $\s$-algebra on $D^\circ(\R,\cD)$ is generated by
the set of all such evaluation maps with $I=(s,t]$ and $s,t$ and $x$ rational.}
                                                                                                                                                                            
{\em For $e=(s,x)\in\R^2$ and $\phi\in D^\circ(\R,\cD)$, the map
$$
t\mapsto\phi_{(s,t]}^+(x):[s,\infty)\to\R
$$
is cadlag.} Hence we can extend the map $Z^{e,+}$, which we defined on $C^\circ(\R,\cD)$ in the
preceding section, to a measurable map
$Z^{e,+}:D^\circ(\R,\cD)\to D_e=D_x([s,\infty),\R)$ by setting
$$
Z^{e,+}(\phi)=(\phi_{(s,t]}^+(x):t\ge s).
$$
Given a countable set $E\sse\R^2$, we then define a measurable map
$Z^{E,+}:D^\circ(\R,\cD)\to D_E$ by
$$
Z^{E,+}(\phi)^e=Z^{e,+}(\phi),
$$
and set 
$$
D_E^{\circ,+}=\{Z^{E,+}(\phi):\phi\in D^\circ(\R,\cD)\}.
$$
Define similarly $Z^{e,-}$, $Z^{E,-}$ and $D_E^{\circ,-}$, and set $D_E^\circ=D_E^{\circ,+}\cap D_E^{\circ,-}$.

\section{Convergence of L\' evy flows to the coalescing Brownian flow}
\label{MR}
We defined in Section \ref{LFC}, for a given basic map $f\in\cD^*$, on the canonical probability
space $(\O,\cF,\PP)$ of a Poisson random measure, a certain
random flow $(X_I:I\sse\R)$, with $X_I\in\cD$ for all finite intervals $I$.
The properties (\ref{LFM}), (\ref{LWF}) and (\ref{LPC}), noted above, imply that
$X(\o)\in D^\circ(\R,\cD)$ for $\PP$-almost all $\o\in\O$, and that the function
$X:\O\to D^\circ(\R,\cD)$ is measurable for the Borel $\s$-algebra of $D^\circ(\R,\cD)$.
Moreover, we know that
\begin{equation}\label{PME}
Z^{e,+}(X)=Z^{e,-}(X)\q\text{almost surely, for all $e\in\R^2$}.
\end{equation}
We denote by $\mu^f_A$ the law of $X$ on the Borel $\s$-algebra of $D^\circ(\R,\cD)$.

\def\comment{It may be possible to base an alternative proof, using a more direct tightness
argument on $D^\circ(\R,\cD)$, on the following estimates, which have remained unused. First
$$
\E(X_t^4)=3t^2\rho\int_0^1\tilde f(z)^2dz+t\rho\int_0^1\tilde f(x)^4dz\le3t^2+4t\rho^{-2/3},
$$
Then
$$
|X_{t0}(x)-x|\le\max_{k=0,1,\dots,n-1}|X_{t0}(k/n)-k/n|+1/n,
$$
so 
$$
2\|d_\cD(X_{t0},\id)\|_{L^p(\PP)}\le(n\E(X_t^p))^{1/p}+1/n.
$$
Then, optimizing over $n$, we get
$$
\|d_\cD(X_{t0},\id)\|_{L^p(\PP)}\le2\E(X_t^p)^{1/(p+1)},
$$
provided the RHS is no greater than $1$. In particular
$$
\|d_\cD(X_{t0},\id)\|_{L^4(\PP)}\le2(3t^2+4t\rho^{-2/3})^{-1/5}.
$$}
\begin{theorem}\label{MAIN}
We have $\mu^f_A\to\mu_A$ weakly on $D^\circ(\R,\cD)$,
uniformly in $f\in\cD^*$, as $\rho(f)\to\infty$ and $\l(f)\to0$.
\end{theorem}
\begin{proof}
Take $E=\Q^2$.
Consider $\phi,\psi\in D^\circ(\R,\cD)$ and suppose that $Z^{E,+}(\phi)=Z^{E,+}(\psi)$. Thus $\phi_{(s,t]}^+(x)=\psi_{(s,t]}^+(x)$ for all
$s,x\in\Q$ and all $t>s$. This extends to all $x\in\R$ because $\phi_{(s,t]}^+$ and $\psi_{(s,t]}^+$ are both right continuous. Then,
using (\ref{PHID}), we obtain $\phi_I(x)=\psi_I(x)$ for all finite intervals $I$.
Hence the map $Z^{E,+}:D^\circ(\R,\cD)\to D^{\circ,+}_E$ is a bijection.
Write $\Phi^{E,+}$ for the inverse bijection $D^{\circ,+}_E\to D^\circ(\R,\cD)$. 
Similarly, write $\Phi^{E,-}$ for the inverse of the bijection $Z^{E,-}:D^\circ(\R,\cD)\to D^{\circ,-}_E$.
Then {\em $D^{\circ,+}_E$ and $\Phi^{E,+}$ are
measurable, and we have $\Phi^{E,+}=\Phi^{E,-}$ on $D^\circ_E$}.
Write $\Phi^E$ for the common restriction of these functions to $D_E^\circ$.
{\em Then $\Phi^E$ is continuous at $z$ for all $z\in C_E^\circ$.}
We see from (\ref{PME}) that $\mu_E^f(D^\circ_E)=1$ for all $f\in\cD^*$, 
and we recall from the proof of Theorem \ref{UBPR} that $\mu_E(C^\circ_E)=1$.

Consider the limit where $\rho(f)\to\infty$ and $\l(f)\to0$.
By Proposition \ref{LAF}, we have $\mu^f_E\to\mu_E$ weakly on $D_E$.
Then, by a standard weak convergence argument, see for example \cite{B}, we obtain 
$\mu^f_A=\mu^f_E\circ(\Phi^E)^{-1}\to\mu_E\circ(\Phi^E)^{-1}=\mu_A$ weakly on $D^\circ(\R,\cD)$.
\end{proof}

\section{Time reversal}\label{TR}
For $f^+\in\cR$ and $f^-\in\cL$, we define a {\em left-continuous inverse}
$(f^+)^{-1}\in\cL$ and a {\em right-continuous inverse} $(f^-)^{-1}\in\cR$
by
$$
(f^+)^{-1}(y)=\inf\{x\in\R:f^+(x)>y\},\q (f^-)^{-1}(y)=\sup\{x\in\R:f^-(x)<y\}.
$$
The map $f^+\mapsto(f^+)^{-1}:\cR\to\cL$ is a bijection, with
$((f^+)^{-1})^{-1}=f^+$ and
$$
(f_1^+\circ f_2^+)^{-1}=(f_2^+)^{-1}\circ(f_1^+)^{-1},\q f_1,f_2\in\cR.
$$
We have $f^+\circ (f^+)^{-1}=\id$\, if and only if $f^+$ is a homeomorphism.
Define for $f=\{f^-,f^+\}\in\cD$ the {\em inverse}
$f^{-1}=\{(f^+)^{-1},(f^-)^{-1}\}\in\cD$. Note that $(f^{-1})^\times=-f^\times$,
so the map $f\mapsto f^{-1}:\cD\to\cD$ is an isometry.

Define the {\em time-reversal map} $\phi\mapsto\hat\phi$ on $D^\circ(\R,\cD)$ by
$$
\hat\phi_I=\phi_{-I}^{-1},
$$
where $-I=\{-x:x\in I\}$.
It is straightforward to check that this is a well-defined isometry of
$D^\circ(\R,\cD)$, which restricts to an isometry of $C^\circ(\R,\cD)$.
Write $\hat\mu_A^f$ for the law of $\widehat{X^f}$, where
$X^f\sim\mu_A^f$ is the L\'evy flow
with basic map $f\in\cD^*$.  
\begin{proposition}
We have $\hat\mu_A^f=\mu_A^{f^{-1}}$ for all $f\in\cD^*$.
\end{proposition}
\begin{proof}
Fix $f\in\cD^*$. Set $g=f^{-1}$ and
$$
\D=\{(x,y)\in\R^2:y<f(x)\}=\{(x,y)\in\R^2:x>g(y)\},\q \D_0=\{(x,y)\in\R^2:y<x\}.
$$
Then, by Fubini's theorem,
$$
\int_0^1\tilde f(x)dx=\int_0^1\int_\R(1_\D-1_{\D_0})(x,y)dxdy=-\int_0^1\tilde g(y)dy
$$
and
$$
\int_0^1\tilde f(x)^2dx=\int_0^1\int_\R2(y-x)(1_\D-1_{\D_0})(x,y)dxdy=\int_0^1\tilde g(y)^2dy.
$$
So $\rho(g)=\rho(f)$ and $\b(g)=-\b(f)$.
Given $\o\in\O$, define $\hat\o\in\O$ by $\hat\o(I\times B)=\o((-I)\times B)$,
for intervals $I\sse\R$ and Borel sets $B\sse(\R/\Z)$. The map $\o\mapsto\hat\o$
is a measure preserving transformation of $(\O,\cF,\PP)$. We have also
$(\hat\o)^\b=\widehat{\o^{-\b}}$. Now, suppressing the $\pm$ notation, for any finite interval
$I$,
$$
X^f_{-I}(\hat\o,x)=\hat F_{\hat T_n}\circ\dots\circ\hat F_{\hat T_n}(x-\b t)+\b s,
$$
where $\hat T_1<\dots<\hat T_n$ are the times of the atoms of $(\hat\o)^\b$
in $-I$, and where
$$
\hat F_t = \begin{cases}
          \t_z f, & \text{if $(\hat\o)^\b$ has an atom at $(t,z)$}, \\
          \id,          & \text{otherwise}.
         \end{cases}
$$
Set
$$
G_t(y) = \begin{cases}
          \t_zg, & \text{if $\o^{-\b}$ has an atom at $(t,z)$}, \\
          \id,          & \text{otherwise}.
         \end{cases}
$$
Note that $\hat T_k=-T_{n-k+1}$, where $T_1<\dots<T_n$ are
the times of the atoms of $\o^{-\b}$ in $I$,
and $(\t_zf)^{-1}=\t_zg$, so $(\hat F_{\hat T_k})^{-1}=G_{T_{n-k+1}}$.
Then
$$
\widehat{X^f_I}(\hat\o,y)=(X^f_{-I})^{-1}(\hat\o,y)
=G_{T_n}\circ\dots\circ G_{T_1}(y-\b s)+\b t=X^g_I(\o,y).
$$
Hence $X^g$ and $\widehat{X^f}$ have the same distribution.
\end{proof}

We get as a corollary the reversibility of the limit, which is already known
in various guises. See, for example, \cite{A79},\cite{FN} and \cite{TW}.

\def\comment{
I have been puzzling over the question whether $\l(f)\to0$ implies $\l(f^{-1})\to0$ in general.
This would tie things up neatly. As it stands, using reversibility, one of these
conditions will suffice, with $\rho(f)\to\infty$, to give $\mu^f_A\to\mu_A$.
I am pretty much convinced that $\l(f)\not=\l(f^{-1})$ in general but have not actually done
an example to establish this.}

\begin{corollary}
The law $\mu_A$ of the coalescing Brownian flow is invariant under time-reversal.
\end{corollary}
\begin{proof}
Write $\hat\mu_A$ for the image measure of $\mu_A$ under time reversal.
Fix $r>0$ and define $f\in\cD$ by
$$
f^+(n+x)=n+(x\vee r),\q n\in\Z,\q x\in[0,1).
$$
Then $\tilde f^+(x)=(r-x)\vee 0$ for $x\in[0,1)$, so $\rho(f)=3/r^3$ and
$$
\int_0^1\tilde f(x)\tilde f(x+a)dx=0,\q r\le a\le1-r,
$$
so $\l(f)\le r$. Moreover $\rho(f^{-1})=\rho(f)$ and $\l(f^{-1})=\l(f)$\note. 
Consider the limit $r\to0$. By Theorem \ref{MAIN},
we know that $\mu^f_A\to\mu_A$ and $\mu^{f^{-1}}_A\to\mu_A$, weakly on $D^\circ(\R,\cD)$.
Since the time-reversal map $\phi\mapsto\hat\phi$ is an isometry, it
follows, using the preceding proposition, that $\mu^{f^{-1}}_A=\hat\mu^f_A\to\hat\mu_A$,
weakly on $D^\circ(\R,\cD)$. Hence $\mu_A=\hat\mu_A$.
\end{proof}

\section{Scaling limit of the aggregation model}\label{SLAM}
\label{model}

We now return to the aggregation model described in the Introduction and deduce from Theorem \ref{LAF}
that, as the particle diameter $\d\to0$ and time is suitably speeded up, then the evolution of 
harmonic measure on the fingers of the cluster converges to the coalescing Brownian flow.
In fact, this scaling limit applies more generally and depends on the shape of the attached particles 
through only a single scalar parameter, which determines the speed of the associated flow.
We therefore begin this section by giving a more general formulation of the aggregation model than that
in the Introduction.

Let $K_0$ denote the closed unit ball in $\C$ with centre at $0$. Set $D_0=(\C\cup\{\infty\})\sm K_0$.
Let $P$ be a closed, connected, simply connected subset of $\overline{D_0}$ of diameter $\d\in(0,1]$ 
such that $P\cap K_0=\{1\}$. Set $D=(\C\cup\{\infty\})\sm(K_0\cup P)$. 
The set $P$ models an incoming particle, which is attached to $K_0$ at $1$. 
Write $F$ for the unique conformal isomorphism $F:D_0\to D$ such that $F(\infty)=\infty$
and $F'(\infty)>0$. Let $R_1,R_2,\dots$ be a sequence of random rotations of the plane about the origin
and set $F_n=R_n\circ F\circ R_n^{-1}$. Now construct the aggregation model from the sequence $F_1,F_2,\dots$
as in the Introduction.

Write $G$ for the inverse isomorphism $D\to D_0$.
There exists a unique $g\in\cD$ such that $g$ restricts to a continuous map from the interval $(0,1)$ to itself, and such that
$$
G(e^{2\pi ix})=e^{2\pi ig(x)},\q x\in(0,1).
$$
Define $\rho(P)=\rho(g)>0$ and $\b(P)=\b(g)\in\R$ as at (\ref{lambdadef}). 
Note that, if $P$ is symmetric in the real axis, then $g$ is an odd function, so $\b(P)=0$.

Set $\G_n=G_n\circ\dots\circ G_1$, where $G_n=F_n^{-1}$, so that $\G_n:D_n\to D_0$.
The restriction of $\G_n$ to the boundary $\partial K_n=\partial D_n$, gives a natural parametrization
of the boundary of the $n$th cluster by the unit circle. It has the property that, for $\xi,\eta\in\partial K_n$,
the normalized harmonic measure $h$ (from $\infty$) of the positively oriented boundary segment from $\xi$ to $\eta$
is given by $\G_n(\eta)/\G_n(\xi)=e^{2\pi ih}$. 
For $m,n\in\N$ with $m<n$, set 
$$
\G_{nm}=G_n\circ\dots\circ G_{m+1}|_{\partial K_0}. 
$$
Set $\G_{nn}=\id$. The circle maps $\G_{nm}$ belong to $\cD_0$ (see Footnote \ref{CDZ})
and have the flow property
$$
\G_{nm}\circ \G_{mk}=\G_{nk},\q k\le m\le n.
$$
The map $\G_{nm}$ expresses how the harmonic measure on $\partial K_m$ is transformed by the arrival of new particles
up to time $n$. The following result identifies the asymptotic behaviour of this family of maps, for symmetric particles $P$,
in the limit as the
particle diameter $\d$ becomes small.\footnote{If $P$ is not symmetric, we obtain the same result once the definition of
$\G_I$ is modified to 
$$
\G_I(e^{2\pi ix})=e^{-2\pi i\b t}\G_{nm}(e^{2\pi i(x+\b s)}),
$$
where $s=\inf I$ and $t=\sup I$.
The cluster exhibits in this case a first order spinning at speed $\b=\b(P)$, with second order fluctuations given by the
coalescing Brownian flow.}
We embed in continuous time by defining, for an interval $I\sse(0,\infty)$,
$\G_I=\G_{nm}$ where $m$ and $n$ are the smallest and largest integers, respectively, in $\rho(P)I$.
Then $(\G_I:I\sse(0,\infty))$ is a random variable in $D^\circ((0,\infty),\cD_0)$. We denote its law by $\mu^P_A$.
Write here $\mu_A$ for the law of the coalescing Brownian flow on $C^\circ((0,\infty),\cD_0)$.

\begin{theorem}
We have $\mu_A^P\to\mu_A$ weakly on $D^\circ((0,\infty),\cD_0)$ as $\d(P)\to0$.
\end{theorem}
\begin{proof}
By a standard argument, it will suffice to prove that $\tilde\mu_A^P\to\mu_A$ weakly on $D^\circ((0,\infty),\cD_0)$ as $\d(P)\to0$,
where we obtain $\tilde\mu^P_A$ by a Poissonized embedding, taking $\tilde\G_{(s,t]}=\G_{N_tN_s}$, with $N$ a Poisson process
of rate $\rho(P)$. In the light of Proposition \ref{GDL} below, the theorem is then a straightforward translation
of Theorem \ref{MAIN}.
\end{proof}

\begin{corollary}
Let $x_1,\dots,x_n$ be a positively oriented set of points in $\R/\Z$ and set $x_0=x_n$.
Set $K_t=K_{\lfloor\rho(P)t\rfloor}$.
For $k=1,\dots,n$, write $H_t^k$ for the harmonic measure in $K_t$ of the boundary segment of all
fingers in $K_t$ attached between $x_{k-1}$ and $x_k$.
Let $(B_t^1,\dots,B_t^n)_{t\ge0}$ be a family
of coalescing Brownian motions in $\R/\Z$ starting from $(x_1,\dots,x_n)$. 
Then, in the limit $\d(P)\to0$, $(H_t^1,\dots,H_t^n)_{t\ge0}$ converges weakly in $D([0,\infty),[0,1]^n)$
to $(B_t^1-B_t^0,\dots,B^n_t-B_t^{n-1})_{t\ge0}$.
\end{corollary}

\begin{proposition}\label{GDL}
There is a universal constant $C<\infty$ such that $\d^{-3}/C \leq \rho(P) \leq C\d^{-3}$  
and $\l(P)\le C\d^{1/4}$.
\end{proposition}
\begin{proof}
It is shown in Lawler \cite{L} that there exists some universal constant $c<\infty$ such that if 
$c\d\leq x\leq1-c\d$, then
$$
|\tilde{g}(x)| \leq \frac{3 \cp(P \cup K_0)}{2\pi|1 - e^{2 \pi ix}|} = \frac{3\cp(P \cup K_0)}{4\pi \sin(\pi x)},
$$
where $\cp(P \cup K_0) = - \log G'(\infty)$ is the logarithmic capacity of $P \cup K_0$. If $K_1 \sse K_2$, 
then $\cp(K_1) \leq \cp(K_2)$ (see, for example, \cite{L}), 
and hence $\cp(P \cup K_0) \le\cp(L\cup K_0) \leq c'\d^2$ 
for some constant $c' > 0$, by the estimate \eqref{LCE}.

Now suppose $x \in (-c\d, c\d)$. Then $g(-c\d) \leq g(x) \leq g(c\d)$ and hence, for some $c''>0$,
$$
|\tilde{g}(x)|
\le|\tilde{g}(c\d)|\vee|\tilde{g}(-c\d)|+2c\d 
\le\frac{3c'\d^2}{4\pi\sin(\pi c\d)}+2c\d
\le c''\d.
$$
Hence, we have
\begin{align*}
\rho(g)^{-1}&=\int_0^1|\tilde{g}(x)|^2 dx
\le2c\d(c''\d)^2+\int_{c\d}^{1-c\d}\frac{9c'^2\d^4}{16\pi^2\sin^2(\pi x)}dx\\
&=2cc''^2\d^3+\frac{9c'^2\d^4}{8\pi^3}\cot(c\pi\d)=O(\d^3).
\end{align*}

To show the other side of the inequality, recall the estimate \eqref{FR3}, which implies
that $\rho(g)\le8\|\tilde{g}\|^{-3}$. So it is enough to show that there exists
some $x \in [0,1)$ with $|\tilde{g}(x)| \geq c'''\d$ for some $c'''>0$.
Define
$$
\tau_A = \inf \{t>0: B_t \in A \},
$$
where $B_t$ is a Brownian motion starting from infinity. Since Brownian motion is invariant under conformal transformations,
$$
g^+(0)-g^-(0)
%=g^+(x^-)-g^-(x^+-1)
=\PP(\t_P\leq\t_{K_0}).
$$
Since $P$ has diameter $\d$, there exist two points, $w, z \in \overline{P}$, 
with $w \in \overline{P} \cap K_0$ and $|w-z| \geq \d/2$. Since $P$ is connected there is a path, 
$l \sse \overline{P}$ joining $w$ and $z$. Let $r$ be the line joining 0 to $z$. Reflect $l$ in $r$ and call the reflection $l'$ and the reflection of point $w$, $w'$. Now
$$
2 \mathbb{P}(\tau_l \leq \tau_{K_0})\geq \mathbb{P}(\tau_{l \cup l'} \leq \tau_{K_0} )
$$
and 
$$
\mathbb{P}(\tau_{l \cup l'} \leq \tau_{K_0} ) \geq \mathbb{P}(\tau_r \leq \tau_{K_0}) \geq 8c'''(|z|-1),
$$
by the limit \eqref{DTHN}, but also
$$
\mathbb{P}(\tau_{l \cup l'} \leq \tau_{K_0}) \geq \mathbb{P}(B_{\tau_{K_0}} \in \text{arc } (w-w')) \geq 8c'''|w-w'|
$$
for some constant $c'''>0$. Either $(|z|-1)$ or $|w-w'| \geq \d/4$ which finishes the argument.

Finally, to show that $\l(g)\le C\d^{1/4}$ observe that for $\d$ sufficiently small, 
if $a \in [\d^{1/4}, 1- \d^{1/4}]$, then at most one of $x$ and $x+a$ lie in 
$\cup_{n \in \mathbb{Z}}[n-c\d, n+c\d]$. Hence
$$
\rho\int_0^1|\tilde{g}(x+a)\tilde{g}(x)|dx
\le2\rho c''\d\int_{c\d}^{1-c\d}\frac{3c'\d^2}{4\pi\sin(\pi x)}dx 
=\frac{3\rho c'c''\d^3}{\pi^2}\log\tan(\pi c\d)
=O(\d).
$$
Hence there exists some $C>0$ such that, for all $a \in [C\d^{1/4}, 1- C\d^{1/4}]$, we have
$$
\rho \int_0^1|\tilde{g}(x+a)\tilde{g}(x)|dx < C\d^{1/4}.
$$
\end{proof}

\section{Appendix}
\subsection{Some properties of the space $\cD$ of non-decreasing functions
of degree $1$}
We give proofs in this subsection of a number of assertions made in Section \ref{CBF}.
\begin{proposition}
\label{fcrossbijection}
The map $f\mapsto f^\times:\cD\to\cS$ is a well-defined bijection, with inverse given by
$$
f^-(x)=\inf\{t+f^\times(t):t\in\R,x=t-f^\times(t)\},\q
f^+(x)=\sup\{t+f^\times(t):t\in\R,x=t-f^\times(t)\}.
$$
\end{proposition}
\begin{proof}
Recall that $f^\times(t)=t-x$, where $x$ is the unique point such that $f^-(x)\le2t-x\le f^+(x)$.
The periodicity of $f^\times$ is an easy consequence of the degree $1$ condition. We now show that
$f^\times$ is a contraction. Fix $s,t\in\R$ and suppose that
$f^\times(s)=s-y$. Switching the roles of $s$ and $t$ if necessary, we may assume without
loss that $x\ge y$. If $x=y$, then $f^\times(s)-f^\times(t)=s-t$. On the other hand, if $x>y$, then
$2s-y\le f^+(y)\le f^-(x)\le2t-x$, so
$$
-(t-s)\le-(t-s)+(2t-x)-(2s-y)=f^\times(t)-f^\times(s)=(t-s)-(x-y)<t-s.
$$
In both cases, we see that $|f^\times(t)-f^\times(s)|\le|t-s|$. Hence $f^\times\in\cS$.

Suppose now that $g\in\cS$. Consider, for each $x\in\R$, the set
$$
I_x=\{t+g(t):t\in\R,x=t-g(t)\}.
$$ 
Since $g$ is a contraction, these sets are all intervals,
and, since $g$ is bounded, they cover $\R$.
For $x,y\in\R$ with $x>y$, and for $s,t\in\R$ with $x=t-g(t),y=s-g(s)$, we have
$t-s-(g(t)-g(s))=x-y>0$, so $s\le t$, and so
$$
t+g(t)-(s+g(s))=t-s+(g(t)-g(s))\ge0.
$$
Define $h^+(y)=\sup I_y$ and $h^-(x)=\inf I_x$.
We have shown that $h^+(y)\le h^-(x)$.
Moreover, since the intervals $I_x$ cover $\R$, the functions $h^\pm$ must be
the left-continuous and right-continuous versions of a non-decreasing function $h$,
which then has the degree $1$ property, because $g$ is periodic.
Thus $h\in\cD$.

For each $t\in\R$, we have $h^\times(t)=t-x$,
where $2t-x\in I_x$, and so $2t-x=s+g(s)$ for some $s\in\R$
with $x=s+g(s)$. Then $s=t$ and so $h^\times(t)=g(t)$. Hence $h^\times=g$.
On the other hand, if we take $g=f^\times$ and if $x$ is a point of continuity of $f$,
then we find $I_x=\{f(x)\}$, so $h^+(x)=h^-(x)=f(x)$. Hence $h=f$.
We have now shown that $f\mapsto f^\times:\cD\to\cS$ is a bijection, and that its inverse
has the claimed form.
\end{proof}

\begin{proposition}
For $f,g\in\cD$ and $\ve>0$,
$$
d_\cD(f,g)\le\ve\iff f^-(x-\ve)-\ve\le g^-(x)\le g^+(x) \le f^+(x+\ve)+\ve
\text{ for all }x\in\R.
$$
Moreover, for any sequence $(f_n:n\in\N)$ in $\cD$,
$$
f_n\to f\text{ in }\cD\q\iff\q f_n^+(x)\to f(x)\text{ at all points $x\in\R$ where $f$ is continuous}.
$$
\end{proposition}
\begin{proof}
Suppose that $d_\cD(f,g)\le\ve$ and that $x$ is a continuity point of $g$.
Then $g(x)=t+g^\times(t)$ for some $t\in\R$ with $x=t-g^\times(t)$.
We must have $x+\ve\ge t-f^\times(t)$ and $g(x)\le t+f^\times(t)+\ve$,
so $f^+(x+\ve)+\ve\ge t+f^\times(t)+\ve\ge g(x)$.
Similarly $f^-(x-\ve)-\ve\le g(x)$. These inequalities extend to all $x\in\R$
by taking left and right limits along continuity points.

Conversely, suppose that $t \in \R$ is such that $|f^\times(t)-g^\times(t)| = d_\cD(f,g)$
and let $x = t-g^\times(t)$ and $y=t-f^\times(t)$. Then $x$ is the unique point with
$g^-(x)+x \leq 2t \leq g^+(x)+x$ and $y$ is the unique point such that
$f^-(y)+y \leq 2t \leq f^+(y)+y$. Hence $f^-(x-\ve)-\ve\le g^-(x)\le g^+(x) \le f^+(x+\ve)+\ve$
implies $y \in [x - \ve, x+ \ve]$ and so $d_\cD(f,g) = |y-x| \le \ve$.

It follows directly that for any sequence $(f_n:n\in\N)$ in $\cD$,
if $d_\cD(f_n,f) \rightarrow 0$ as $n \rightarrow \infty$,
then $f_n^+(x)\to f(x)$ at all points $x\in\R$ where $f$ is continuous.

Now suppose $f_n^+(x)\to f(x)$ at all points $x\in\R$ where $f$ is continuous.
By equicontinuity, it will suffice to show that $f_n^\times(t) \rightarrow f^\times(t)$ for each $t\in\R$.
Set $x=t-f^\times(t)$ and $x_n=t-f^\times_n(t)$.
Given $\ve>0$, choose $y_1\in(x-\ve,x)$ and $y_2\in(x,x+\ve)$, both points
of continuity of $f$. Now $f(y_1)+y_1<2t<f(y_2)+y_2$, so there exists $N\in\N$
such that for all $n \geq N$, we have $f^+_n(y_1)+y_1<2t<f^+_n(y_2)+y_2$, which implies
$x_n\in[y_1,y_2]$ and hence $|f_n^\times(t) - f^\times(t)| < \ve$, as required.
\end{proof}
\begin{proposition}\label{WFL}
Suppose $f_n\to f, g_n\to g, h_n\to h$ in $\cD$ with
$h_n^+\le f_n^+\circ g_n^+$ for all $n$. Then $h^+\le f^+\circ g^+$.
\end{proposition}
\begin{proof}
It will suffice to establish the inequality at all
points $x$ where $g$ and $h$ are both continuous.
Given $\ve>0$, since $f^+$ is right-continuous, there exists a point $y>g(x)$
where $f$ is continuous and such that $f(y)<f^+(g(x))+\ve$.
Then $f_n^+(y) < f^+(g(x))+\ve$ and $g_n^+(x)\le y$ eventually, so
$$
h_n^+(x)\le f_n^+(g_n^+(x))\le f_n^+(y)<f^+(g(x))+\ve
$$
eventually. Hence $h^+(x)=\lim_{n\to\infty}h_n^+(x)\le f^+(g^+(x))$,
as required.
\end{proof}

\subsection{Some properties of the the continuous flow-space $C^\circ(\R,\cD)$
and cadlag flow-space $D^\circ(\R,\cD)$}
We give proofs in this subsection of a number of assertions made in Sections \ref{CBF} and \ref{SKOR}.
\begin{proposition}
For $(s,x)\in\R^2$ and $\phi\in D^\circ(\R,\cD)$, the map
$$
t\mapsto\phi_{(s,t]}^+(x):[s,\infty)\to\R
$$
is cadlag, and is moreover continuous whenever $\phi\in C^\circ(\R,\cD)$.
\end{proposition}
\begin{proof}
Given $t\ge s$ and $\ve>0$, we can choose $\d>0$ so that $d_\cD(\phi_{(t,u]},\id)<\ve/2$ for all $u\in(t,t+\d]$.
For such $u$ and for $x$ a point of continuity of $\phi_{(s,t]}$, we have
$$
\phi_{(s,t]}^+(x)-\ve=\phi_{(s,t]}^-(x)-\ve\leq\phi_{(t,u]}^-\circ\phi_{(s,t]}^-(x)\leq 
\phi_{(s,u]}^-(x)\leq 
\phi_{(s,u]}^+(x) \leq \phi_{(t,u]}^+ \circ \phi_{(s,t]}^+(x) \leq \phi_{(s,t]}^+(x) +  \ve,
$$
so $|\phi_{(s,u]}^+(x)-\phi_{(s,t]}^+(x)|\leq\ve$. The final estimate extends to all $x$ by right-continuity.
Hence the map is right continuous. A similar argument shows that, for $u\in(s,t)$, we have
$|\phi_{(s,u]}^+(x)-\phi_{(s,t)}^+(x)|\to0$ as $u\to t$, so that the map has a left limit at $t$
given by $\phi_{(s,t)}^+(x)$. Finally, if $\phi\in C^\circ(\R,\cD)$, then
$\phi_{(s,t)}=\phi_{(s,t]}$, so the map is continuous.
\end{proof}

\begin{proposition}\label{PHICONT}
For all $\phi\in C^\circ(\R,\cD)$, the map $(s,t)\mapsto\phi_{ts}:\{(s,t):s\le t\}\to\cD$ is continuous.
Moreover, for all $\phi\in D^\circ(\R,\cD)$ and for any sequence of finite intervals $I_n\to I$, we have $\phi_{I_n}\to\phi_I$.
\end{proposition}
\begin{proof}
The first assertion follows from the second: given $\phi\in C^\circ(\R,\cD)$ and sequences $s_n\to s$ and $t_n\to t$, then,
passing to a subsequence if necessary, we can assume that $(s_n,t_n]\to I$ for some interval $I$ with $\inf I=s$ and $\sup I=t$.
Then, by the second assertion, we have $\phi_{t_ns_n}\to\phi_I=\phi_{ts}$, as required.

So, let us fix $\phi\in D^\circ(\R,\cD)$ and a sequence of finite intervals $I_n\to I$.
By combining the cadlag and weak flow properties, we can show the following variant of the cadlag property: for all $t\in\R$, we have 
\begin{equation}\label{WFP2}
\phi_{[s,t)}\to\id\q\text{as $s\uparrow t$},\q
\phi_{(t,u]}\to\id\q\text{as $u\downarrow t$}.
\end{equation}
For each $n$, there exist two disjoint intervals $J_n$ and $J_n'$, possibly empty,
such that $I\triangle I_n=J_n\cup J_n'$. For any such $J_n$ and $J_n'$, using the weak flow property, we obtain
$$
d_\cD(\phi_I,\phi_{I_n})\le\|\phi_{J_n}-\id\|+\|\phi_{J_n'}-\id\|.
$$
Set $s=\inf I$, $s_n=\inf I_n$, $t=\sup I$ and $t_n=\sup I_n$.
Then $s_n\to s$, $t_n\to t$, and
$$
\text{ if $s\in I$ then $s\in I_n$ eventually},\q
\text{ if $s\not\in I$ then $s\not\in I_n$ eventually},
$$$$
\text{ if $t\in I$ then $t\in I_n$ eventually},\q
\text{ if $t\not\in I$ then $t\not\in I_n$ eventually}.
$$
Hence, using the cadlag property or (\ref{WFP2}), or both, we find that
$\phi_{J_n}\to\id$ and $\phi_{J_n'}\to\id$, which proves the proposition.
\end{proof}

\begin{proposition}
The metrics $d_C$ and $d_D$ generate the same topology on $C^\circ(\R,\cD)$.
\end{proposition}
\begin{proof}
On comparing the definitions of $d^C_n$ and $d^D_n$ for each $n\in\N$, and considering the choice $\l=\id$,
we see that $d_D\le d_C$. Hence, it will suffice to show, given $\phi\in C^\circ(\R,\cD)$, $n\in\N$ and $\ve>0$,
that there exists $\ve'>0$ such that, for all $\psi\in C^\circ(\R,\cD)$, we have $d_C^{(n)}(\phi,\psi)<\ve$
whenever $d_{n+1}^D(\phi,\psi)<\ve'$. By the preceding proposition, there exists a $\d\in(0,1]$ such that
$d_\cD(\phi_{ts},\phi_{t's'})<\ve/2$ whenever $|s-s'|,|t-t'|\le\d$ and $s,t\in(-n,n)$. Set $\ve'=\d\wedge(\ve/2)$
and suppose that $d_{n+1}^D(\phi,\psi)<\ve'$. Then there exists an increasing homeomorphism $\l$ of $\R$, with
$|\l(t)-t|\le\d$ for all $t$, such that, for all intervals $I$, we have 
$\|\chi_{n+1}(I)\psi_I^\times-\chi_{n+1}(\l(I))\phi_{\l(I)}^\times\|<\ve/2$.
Given $s,t\in(-n,n)$ with $s<t$, take $I=(s,t]$.
Then $\chi_{n+1}(I)=\chi_{n+1}(\l(I))=1$, so $d_\cD(\phi_{\l(t)\l(s)},\psi_{ts})=\|\psi_I^\times-\phi_{\l(I)}^\times\|<\ve/2$.
But then, for all such $s,t$, we have 
$$
d_\cD(\phi_{ts},\psi_{ts})\le d_\cD(\phi_{ts},\phi_{\l(t)\l(s)})+d_\cD(\phi_{\l(t)\l(s)},\psi_{ts})<\ve,
$$
so $d_C^{(n)}(\phi,\psi)<\ve$, as required.
\end{proof}

\begin{proposition}
The metric spaces
$(C^\circ(\R,\cD),d_C)$ and $(D^\circ(\R,\cD),d_D)$ are complete and separable.
\end{proposition}
\begin{proof}
The argument for completeness is a variant of the corresponding
argument for the usual Skorokhod space $D(\R, S)$ of cadlag paths in complete separable metric
space $S$, as found for example in \cite{B}.
Suppose then that $(\psi^n)_{n \geq 1}$ is a Cauchy sequence in $D^\circ(\R,\cD)$.
There exists a subsequence $\phi^k=\psi^{n_k}$ such that $d_D^{(n)}(\phi^n, \phi^{n+1}) < 2^{-n}$ for all $n\ge1$.
It will suffice to find a limit in $D^\circ(\R,\cD)$ for $(\phi^n)_{n \geq 1}$.
There exist increasing homeomorphisms $\mu_n$ of $\R$
for which $\gamma(\mu_n) < 2^{-n}$ and
$$
d_\cD(\phi^n_I, \phi^{n+1}_{\mu_n(I)})<2^{-n},\q I\cup\mu_n(I)\sse(-n,n).
$$
For each $n\ge1$, the sequence $(\mu_{n+m} \circ \cdots \circ \mu_{n})_{m\ge1}$ converges
uniformly on $\R$ to an increasing homeomorphism, $\lambda_n$ say,
with $\gamma(\lambda_n) < 2^{-n+1}$. Then $\mu_n\circ\l_n^{-1}=\l_{n+1}^{-1}$, so
$$
d_\cD(\phi^n_{\lambda_n^{-1}(I)}, \phi^{n+1}_{\lambda_{n+1}^{-1}(I)}) < 2^{-n},\q
I\sse(-n+1,n-1).
$$
So, for all $m\ge n$,
\begin{equation}\label{PHMN}
d_\cD(\phi^n_{\lambda_n^{-1}(I)}, \phi^{n+m}_{\lambda_{n+m}^{-1}(I)}) < 2^{-n+1},\q I\sse(-n+1,n-1).
\end{equation}
Hence, for all finite intervals $I\sse\R$,
$(\phi^n_{\lambda_n^{-1}(I)})_{n \geq 1}$ is a
Cauchy sequence in $\cD$, which, since $\cD$ is complete, has a limit $\phi_I\in\cD$.
On letting $m\to\infty$ in (\ref{PHMN}), we obtain
$$
d_\cD(\phi^n_{\lambda_n^{-1}(I)},\phi_I) < 2^{-n+1},\q I\sse(-n+1,n-1).
$$
By Proposition \ref{WFL}, $\phi=(\phi_I:I\sse\R)$ has the weak flow property.
To see that $\phi$ is cadlag, suppose given $\ve>0$ and $t\in\R$. Choose $n$ such that
$2^{-n+1}\le\ve/3$ and $|t|\le n-2$. Then choose $\d\in(0,1]$ such that
$$
d_\cD(\phi^n_{\l_n^{-1}(s,t)},\id)<\ve/3,\q d_\cD(\phi^n_{\l_n^{-1}(t,u)},\id)<\ve/3
$$
whenever $s\in(t-\d,t)$ and $u\in(t,t+\d)$. For such $s$ and $u$, we then have
$$
d_\cD(\phi_{(s,t)},\id)<\ve,\q d_\cD(\phi_{(t,u)},\id)<\ve.
$$
Hence $\phi\in D^\circ(\R,\cD)$. For $m\le n-3$, we have
\begin{align*}
d_D^{(m)}(\phi^n,\phi)
&\le\g(\l_n)\vee\sup_{I\sse(-m-2,m+2)}\|\chi_m(\l_n^{-1}(I))\phi_{\l_n^{-1}(I)}^{n\times}-\chi_m(I)\phi_I^\times\|\\
&\le\g(\l_n)\vee\sup_{I\sse(-m-2,m+2)}\left\{d_\cD(\phi^n_{\lambda_n^{-1}(I)},\phi_I)+\g(\l_n)\|\phi_I^\times\|\right\}\\
&\le2^{-n+1}(1+\sup_{I\sse(-m-2,m+2)}\|\phi_I^\times\|).
\end{align*} 
Hence $d_D(\phi^n,\phi)\to0$ as $n\to\infty$.
We have shown that $D^\circ(\R,\cD)$ is complete.
If the sequence $(\phi^n)_{n \geq 1}$ in fact lies in $C^\circ(\R,\cD)$, then,
by an obvious variation of the argument for the cadlag property, the limit
$\phi$ also lies in $C^\circ(\R,\cD)$. Hence $C^\circ(\R,\cD)$ is also complete.
In particular, $C^\circ(\R,\cD)$ is a closed subspace in $D^\circ(\R,\cD)$.

We turn to the question of separability.
Let us write $D_N$ for the set of those $\phi\in D^\circ(\R,\cD)$ such that:
\begin{itemize}
\item[(i)] for some $n\in\N$ and some rationals $t_1<\dots<t_n$, we have
$\phi_J=\id$ for all time intervals $J$ which do not intersect the set $\{t_1,\dots,t_n\}$;
\item[(ii)] for all other time intervals $I$, the maps $\phi_I$ and $\phi_I^{-1}$ on $\R$
are constant on all space intervals which do not intersect $2^{-N}\Z$.
\end{itemize}
Note that each $\phi\in D_N$ is determined by the maps $\phi_{(t_k,t_m]}$,
for integers $0\le k<m\le n$, where $t_0<t_1$,
and for each of these maps there are only countably many possibilities (finitely many
if we insist that $\phi(0)\in[0,1)$).
Hence $D_N$ is countable and so is $D_*=\bigcup_{N\ge1}D_N$.
We shall show that $D_*$ is also dense in $D^\circ(\R,\cD)$.

Fix $\phi\in D^\circ(\R,\cD)$ and $n_0\ge1$. It will suffice to find, for a given $\ve>0$,
a $\psi\in D_*$ with $d^{(n_0)}_D(\phi,\psi)<\ve$. By the cadlag property and compactness,
there exist $n\in\N$ and reals $s_1<\dots<s_n$ in $I_0=(-n_0-1,n_0+1)$ such that
$d_\cD(\phi_I,\id)<\ve/4$ for every subinterval $I$ of $I_0$ which does not intersect $\{s_1,\dots,s_n\}$.
Then we can find rationals $t_1<\dots<t_n$ in $I_0$ and
an increasing homeomorphism $\l$ of $\R$, with $\l(t)=t$ for $t\not\in I_0$, with $\g(\l)\sup_{I\sse I_0}\|\phi^\times_I\|<\ve/4$,
and such that $\l(t_m)=s_m$ for all $m$.
Set $s_0=t_0=-n_0-1$.

For $f\in\cD$, write $\Delta(f)$ for the set of points where $f$ is not continuous.
Define, for $m=0,1,\dots,n$,
$$
\Delta_m=\bigcup_{k=0}^{m-1}\Delta(\phi_{(s_k,s_m]}^{-1})\cup\bigcup_{k=m+1}^n\Delta(\phi_{(s_m,s_k]}).
$$
Then $\Delta_m$ is countable, so we can choose $N\ge1$ with $16.2^{-N}\le\ve$ and choose
$\ve_m\in\R$ with $|\ve_m|\le2^{-N}$ such that
$$
\t_m(\Delta_m)\cap2^{-N}\Z=\es,\q m=0,1,\dots,n,
$$
where $\t_m(x)=x+\ve_m$. Set
$$
\d^-(x)=2^N\lceil2^{-N}x\rceil,\q \d^+(x)=2^N\lfloor2^{-N}x\rfloor+1.
$$
Note that $\d=\{\d^-,\d^+\}\in\cD$. Define for $0\le k<m\le n$
$$
\psi_{(t_k,t_m]}^-=(\d^{-1})^-\circ(\t_m)^{-1}\circ\phi_{(s_k,s_m]}^-\circ\t_k\circ\d^-,\q
\psi_{(t_k,t_m]}^+=(\d^{-1})^+\circ(\t_m)^{-1}\circ\phi_{(s_k,s_m]}^+\circ\t_k\circ\d^+.
$$
Then $\psi_{(t_k,t_m]}=\{\psi_{(t_k,t_m]}^-,\psi_{(t_k,t_m]}^+\}\in\cD$
by our choice of $\ve_k$ and $\ve_m$.
Moreover $\d^+\circ(\d^{-1})^+\ge\id$ and $\d^-\circ(\d^{-1})^-\le\id$ so,
for $0\le m<m'<m''\le n$, we obtain the inequalities
$$
\psi_{(t_{m'},t_{m''}]}^-\circ\psi_{(t_m,t_{m'}]}^-\le\psi_{(t_m,t_{m''}]}^-
\le\psi_{(t_m,t_{m''}]}^+\le\psi_{(t_{m'},t_{m''}]}^+\circ\psi_{(t_m,t_{m'}]}^+
$$
from the corresponding inequalities for $\phi$.
We use the equations $\|\d-\id\|=2^{-N}$ and $\|\t_m-\id\|=|\ve_m|$
to see that
$$
d_\cD(\phi_{(s_k,s_m]},\psi_{(t_k,t_m]})\le4.2^{-N},\q 0\le k<m\le n.
$$
Define $\psi_J=\psi_{(t_k,t_m]}$ for all intervals $J$ such that
$J\cap\{t_1,\dots,t_n\}=\{t_{k+1},\dots,t_m\}$.
For such intervals $J$, with $J\sse I_0$, we have
$d_\cD(\phi_{(s_k,s_m]\sm\l(J)},\id)<\ve/4$ and $d_\cD(\phi_{\l(J)\sm(s_k,s_m]},\id)<\ve/4$ so,
using the weak flow property for $\phi$,
$$
d_\cD(\psi_J, \phi_{\l(J)})\le d_\cD(\psi_{(t_k,t_m]},\phi_{(s_k,s_m]})
+d_\cD(\phi_{(s_k,s_m]},\phi_{\l(J)})\le 4.2^{-N}+2\ve/4<3\ve/4.
$$
Define $\psi_J=\id$ for all intervals $J$ which do not intersect $\{t_1,\dots,t_n\}$.
For such intervals $J$ with $J\sse I_0$, we have $d_\cD(\psi_J, \phi_{\l(J)})\le d_\cD(\id,\phi_{\l(J)})\le\ve/4$.
Now $\psi\in D_N$ and 
$$
d^{(n_0)}_D(\phi,\psi)\le\g(\l)\vee\sup_{J\sse I_0}\{d_\cD(\psi_J,\phi_{\l(J)})+\g(\l)\|\phi_J^\times\|\}<\ve, 
$$
as required.
This proves that $D^\circ(\R,\cD)$ is separable and, since $C^\circ(\R,\cD)$
is a closed subspace of $D^\circ(\R,\cD)$, it follows that $C^\circ(\R,\cD)$
is also separable.
\end{proof}

\begin{proposition}
For all $s,t\in\R$ with $s<t$, and all $x\in\R$, the map
$\phi\mapsto\phi_{ts}^+(x)$ on $C^\circ(\R,\cD)$ is continuous.
Moreover the Borel $\s$-algebra on $C^\circ(\R,\cD)$ is generated by
the set of all such maps with $s,t$ and $x$ rational.

For all finite intervals $I\sse\R$ and all $x\in\R$, the 
map $\phi\mapsto\phi_I^+(x)$ on $D^\circ(\R,\cD)$ is Borel 
measurable, and the Borel $\s$-algebra on $D^\circ(\R,\cD)$ is generated by
the set of all such maps where $I=(s,t]$ with $s,t$ and $x$ rational.
\end{proposition}
\begin{proof}
The results for $C^\circ(\R,\cD)$ can be proved more simply than those for $D^\circ(\R,\cD)$.
We omit details of the former, but note that these follow also from the latter, by general 
measure theoretic arguments, given what we already know about the two spaces. 

The proof for $D^\circ(\R,\cD)$ is an adaptation of the analogous result for the classical Skorokhod 
space, see for example \cite[page 335]{MR1943877}.
We prove first the Borel measurability of the evaluation maps.
Given a finite interval $I$ and $x\in\R$, we can find $s_n,t_n\in\R$ such that $(s_n,t_n]\to I$ as $n\to\infty$.
Then $\phi^+_I(x)=\lim_{m\to\infty}\lim_{n\to\infty}\phi^+_{(s_n,t_n]}(x+1/m)$, by Proposition \ref{PHICONT}.
Hence, it will suffice to consider intervals $I$ of the form $(s,t]$. Fix $s,t$ and $x$ and define for each $m,n\in\N$
a function $F_{m,n}$ on $D^\circ(\R,\cD)$ by 
$$
F_{m,n}(\phi)=\int_s^{s+1/n}\int_t^{t+1/n}\int_x^{x+1/m}\phi^+_{(s',t']}(x')dx'dt'ds'.
$$
Suppose $\phi^k\to\phi$ in $D^\circ(\R,\cD)$. We can choose increasing homeomorphisms $\l_k$ of $\R$ such that,
$\g(\l_k)\to0$ and, uniformly in $r\in[s-1,s+1]$ and $u\in[t-1,t+1]$, we have
$$
d_\cD(\phi^k_{\l_k(r,u]},\phi_{(r,u]})\to0.
$$
Define
$$
f(r,u)=\int_x^{x+1/m}\phi_{\l(r,u]}(x')dx',\q
f_k(r,u)=\int_x^{x+1/m}\phi^k_{\l_k(r,u]}(x')dx'.
$$
Then $f_k(r,u)\to f(r,u)$, uniformly in $r\in[s-1,s+1]$ and $u\in[t-1,t+1]$.
Set $\mu_k=\l_k^{-1}$. Then
\begin{align*}
F_{m,n}(\phi^k)&=\int_{\mu_k(s)}^{\mu_k(s+1/n)}\int_{\mu_k(t)}^{\mu_k(t+1/n)}f_k(r,u)d\l_k(u)d\l_k(r)\\
&\to\int_s^{s+1/n}\int_t^{t+1/n}f(r,u)dudr= F_{m,n}(\phi),
\end{align*}
so $F_{m,n}$ is continuous on $D^\circ(\R,\cD)$. By Proposition \ref{PHICONT}, we have 
$$
\phi_{(s,t]}^+(x)=\lim_{m\to\infty}\lim_{n\to\infty}\frac1{mn^2}F_{m,n}(\phi).
$$
Hence $\phi\mapsto\phi_{(s,t]}^+(x)$ is Borel measurable, as required.

Write now $\cE$ for the $\s$-algebra on $D^\circ(\R,\cD)$ generated by all maps of this form
with $s,t$ and $x$ rational. 
It remains to show that $\cE$ contains the Borel $\s$-algebra of $D^\circ(\R,\cD)$.
Write $\{(I_k,z_k):k\in\N\}$ for an enumeration of the set $\{(s,t]:s,t\in\Q,s<t\}\times\Q$.
It is straightforward to show that, for all $k$, the map $\phi\mapsto\phi_{I_k}^\times(z_k)$ is $\cE$-measurable.
Fix $n\in\N$, $\phi^0\in D^\circ(\R,\cD)$, $r\in(0,\infty)$ and $k\in\N$, and consider the set
$$
A(k,r)=\{\phi\in D^\circ(\R,\cD):(\chi_n(I_1)\phi_{I_1}^\times(z_1),\dots,\chi_n(I_k)\phi_{I_k}^\times(z_k))\in B(k,r)\},
$$
where
$$
B(k,r)=\bigcup_\l\{(y_1,\dots,y_k)\in\R^k:\max_{j\le k}|y_j-\chi_n(\l(I_j))\phi_{\l(I_j)}^{0\times}(z_j)|<r\},
$$
where the union is taken over all increasing homeomorphisms $\l$ of $\R$ with $\g(\l)<r$.
Note that $B(k,r)$ is an open set in $\R^k$, so $A(k,r)\in\cE$, so $A=\bigcup_{m\in\N}\bigcap_{k\in\N}A(k,r-1/m)\in\cE$.

Consider the set 
$$
C=\{\phi\in D^\circ(\R,\cD):d_D^{(n)}(\phi,\phi^0)<r\}.
$$
It is straightforward to check from the definition of $d_D^{(n)}$, that $C\sse A$. Suppose that $\phi\in A$. We shall
show that $\phi\in C$. Then $C=A$, so $C\in\cE$, and since sets of this form generate the Borel $\s$-algebra, we are done.

We can find an $m\in\N$ and, for each $k\in\N$, a $\l_k$ with $\g(\l_k)<r-1/m$ such that 
$$
|\chi_n(I_j)\phi_{I_j}^\times(z_j)-\chi_n(\l_k(I_j))\phi_{\l_k(I_j)}^{0\times}(z_j)|<r-1/m,\q j=1,\dots,k.
$$
Without loss of generality, we may assume that the sequence $(\l_k:k\in\N)$ converges uniformly on compacts, and 
that its limit, $\l$ say, satisfies $\g(\l)\le r-1/m$.
By Proposition \ref{PHICONT}, for each $j$, there is an interval $\hat I_j$, having the same endpoints as $I_j$
such that $\phi_{\l(\hat I_j)}$ is a limit point in $\cD$ of the sequence $(\phi_{\l_k(I_j)}:k\in\N)$,
so $\phi_{\l(\hat I_j)}^\times$ is a limit point in $\cS$ of the sequence $(\phi_{\l_k(I_j)}^\times:k\in\N)$. Then
$$
|\chi_n(I_j)\phi_{I_j}^\times(z_j)-\chi_n(\l(\hat I_j))\phi_{\l(\hat I_j)}^{0\times}(z_j)|\le r-1/m,
$$
for all $j$. For all finite intervals $I$ and all $z\in\R$, we can find a sequence $(j_p:p\in\N)$ such that
$I_{j_p}\to I$, $\hat I_{j_p}\to I$ and $z_{j_p}\to z$. So we obtain
$$
|\chi_n(I)\phi_{I}^\times(z)-\chi_n(\l(I))\phi_{\l(I)}^{0\times}(z)|\le r-1/m.
$$
Hence $d_D^{(n)}(\phi,\phi^0)\le r-1/m$ and $\phi\in C$, as we claimed.
\end{proof}

Recall that we define, for $e=(s,x)\in\R^2$ and $\phi\in D^\circ(\R,\cD)$,
$$
Z^{e,+}(\phi)=(\phi^+_{(s,t]}(x):t\ge s)
$$
and for $E\sse\R^2$, we define $Z^{E,+}:D^\circ(\R,\cD)\to D_E$ by $Z^{E,+}(\phi)^e=Z^{e,+}(\phi)$.
Recall also that we set
$$
C^{\circ,+}_E=\{Z^{E,+}(\phi):\phi\in C^\circ(\R,\cD)\},\q
D^{\circ,+}_E=\{Z^{E,+}(\phi):\phi\in D^\circ(\R,\cD)\},
$$
and that we define analogously $Z^{E,-}$, $C^{\circ,-}_E$ and $D^{\circ,-}_E$, 
and we set $C_E^\circ=C^{\circ,+}_E\cap C^{\circ,-}_E$ and $D_E^\circ=D^{\circ,+}_E\cap D^{\circ,-}_E$.

\begin{proposition}
Let $E$ be a countable subset of $\R^2$ containing $\Q^2$.
Then $Z^{E,+}:C^\circ(\R,\cD)\to C^{\circ,+}_E$ is a bijection,
$C^{\circ,+}_E$ is a measurable subset of $C_E$, and
the inverse bijection $\Phi^{E,+}:C^{\circ,+}_E\to C^\circ(\R,\cD)$ is a measurable map. 

Moreover $Z^{E,+}:D^\circ(\R,\cD)\to D^{\circ,+}_E$ is also a bijection,
$D^{\circ,+}_E$ is a measurable subset of $D_E$
and the inverse bijection $\Phi^{E,+}:D^{\circ,+}_E\to D^\circ(\R,\cD)$ is also a measurable map. 

Moreover the same statements hold with $+$ replaced by $-$, and we have $\Phi^{E,+}=\Phi^{E,-}$ on $D_E^\circ$.
\end{proposition}
\begin{proof}
We discuss only the cadlag case. The usual comments apply concerning the relationship of this
case with the continuous case. The simple argument that $Z^{E,+}$ is injective on $D^{\circ,+}_E$ 
was given in the proof of Theorem \ref{MAIN}.
Consider for $z\in D_E$ the conditions
\begin{equation}\label{ZDEG1}
z_t^{(s,x+n)}=z_t^{(s,x)}+n,\q s,t,x\in\Q,\q s<t,\q n\in\Z
\end{equation}
and
\begin{equation}\label{ZRIGHT}
z_t^{(s,x)}=\inf_{y\in\Q,y>x}z_t^{(s,y)},\q (s,x)\in E,\q t\in\Q,t>s.
\end{equation}
Under these conditions, define for $s,t\in\Q$ with $s<t$ and for $x\in\R$,
$$
\Phi_{(s,t]}^-(x)=\sup_{y\in\Q,y<x}z_t^{(s,y)},\q
\Phi_{(s,t]}^+(x)=\inf_{y\in\Q,y>x}z_t^{(s,y)}.
$$
Then $\Phi_{(s,t]}=\{\Phi_{(s,t]}^-,\Phi_{(s,t]}^+\}\in\cD$ and
$$
\Phi_{(s,t]}^+(x)=z_t^{(s,x)},\q s,t,x\in\Q,\q s<t.
$$
Now consider the following additional conditions on $z$: 
\begin{equation}\label{ZFLOW}
\Phi_{(t,u]}^-\circ\Phi_{(s,t]}^-\le\Phi_{(s,u]}^-\le\Phi_{(s,u]}^+\le\Phi_{(t,u]}^+\circ\Phi_{(s,t]}^+,\q s,t,u\in\Q,\q s<t<u;
\end{equation}
and
\begin{itemize}
\item[]for all $\ve>0$ and all $n\in\N$, there exist $\d>0$, $m\in\Z^+$ and $u_1,\dots,u_m\in(-n,n)$ such that
\end{itemize}
\begin{equation}\label{ZCADLAG}
\|\Phi_{(s,t]}-\id\|<\ve 
\end{equation}
\begin{itemize}
\item[]whenever $s,t\in\Q\cap(-n,n)$ with $0<t-s<\d$ and $(s,t]\cap\{u_1,\dots,u_m\}=\es$.
\end{itemize}
Note that the inequalities between functions required in (\ref{ZFLOW}) hold whenever the same inequalities
hold between their restrictions to $\Q$, by left and right continuity.
Note also that condition (\ref{ZCADLAG}) is equivalent to the following condition involving quantifiers only over countable sets:
\begin{itemize}
\item[]for all rationals $\ve>0$ and all $n\in\N$, there exist a rational $\d>0$ and an $m\in\Z^+$ such that, for all rationals $\eta>0$,
there exist rationals $s_1,t_1,\dots,s_m,t_m\in(-n,n)$, with $s_i<t_i$ for all $i$ and with
$\sum_{i=1}^m(t_i-s_i)<\eta$, such that
\end{itemize}
$$
\|\Phi_{(s,t]}-\id\|<\ve
$$
\begin{itemize}
\item[]whenever $s,t\in\Q\cap(-n,n)$ with $0<t-s<\d$ and $(s,t]\cap((s_1,t_1]\cup\dots\cup(s_m,t_m])=\es$.
\end{itemize}
Denote by $D_E^{*,+}$ the set of those $z\in D_E$ where conditions (\ref{ZDEG1}), (\ref{ZRIGHT}), (\ref{ZFLOW}) and (\ref{ZCADLAG})
all hold. Then $D^{*,+}_E$ is a measurable subset of $D_E$. Fix $z\in D^{*,+}_E$. Given finite interval $I$, we can find
sequences of rationals $s_n$ and $t_n$ such that $(s_n,t_n]\to I$ as $n\to\infty$. Then, by conditions (\ref{ZFLOW}) and (\ref{ZCADLAG}),
$$
d_\cD(\Phi_{(s_n,t_n]},\Phi_{(s_m,t_m]})\le\|\Phi_{(s_n,s_m]}-\id\|+\|\Phi_{(t_n,t_m]}-\id\|\to0
$$
as $n,m\to\infty$. So the sequence $\Phi_{(s_n,t_n]}$ converges in $\cD$, with limit $\Phi_I$, say, and $\Phi_I$ does not
depend on the approximating sequences of rationals.
In the case where $I=I_1\oplus I_2$, there exists another sequence of rationals $u_n$ such that $(s_n,u_n]\to I_1$
and $(u_n,t_n]\to I_2$ as $n\to\infty$. Hence $\Phi=(\Phi_I:I\sse\R)$ has the weak flow property, by Proposition
\ref{WFL}. It is straightforward to deduce from (\ref{ZCADLAG}) that $\Phi$ is moreover cadlag, so $\Phi=\Phi(z)\in D^\circ(\R,\cD)$.
It follows from its construction, and the preceding proposition, that the map $z\mapsto\Phi(z):D^{*,+}_E\to D^\circ(\R,\cD)$ is measurable.

Now, for all $z\in D^{*,+}_E$, we have $Z^{E,+}(\Phi(z))=z$ and for all $\phi\in D^\circ(\R,\cD)$, we have $Z^{E,+}(\phi)\in D^*_E$ and
$\Phi(Z^{E,+}(\phi))=\phi$. Hence $D_E^{\circ,+}=D^{*,+}_E$ and $\Phi^{E,+}=\Phi$, showing that these are measurable, as claimed.

Consider now for $z\in D_E$ the condition
\begin{equation}\label{ZLEFT}
z_t^{(s,x)}=\sup_{y\in\Q,y<x}z_t^{(s,y)},\q (s,x)\in E,\q t\in\Q, t>s.
\end{equation}
Denote by $D_E^{*,-}$ the set of those $z\in D_E$ where conditions (\ref{ZDEG1}), (\ref{ZFLOW}), (\ref{ZCADLAG}) and 
(\ref{ZLEFT}) all hold, and define $\Phi$ on $D^{*,-}_E$ exactly as on $D^{*,+}_E$. 
Then, by a similar argument, $D_E^{\circ,-}=D^{*,-}_E$ and $\Phi=\Phi^{E,-}$ on $D_E^{\circ,-}$.
In particular $\Phi^{E,+}=\Phi^{E,-}$ on $D_E^\circ$, as claimed.
\end{proof}

\begin{proposition}\label{MPI}
Let $E$ be a countable subset of $\R^2$ containing $\Q^2$.
Then $\mu_E(C^\circ_E)=1$.
\end{proposition}
\begin{proof}
We use an identification of $C^\circ_E$ analogous to that implied for $D^\circ_E$ by the preceding proof.
The same five conditions (\ref{ZDEG1}), (\ref{ZRIGHT}), (\ref{ZFLOW}), (\ref{ZCADLAG}) and (\ref{ZLEFT}) characterize
$C^\circ_E$ inside $C_E$, except that, in (\ref{ZCADLAG}), only the case $m=0$ is allowed.
Recall that, under $\mu_E$, for time-space starting points $e=(s,x)$ and $e'=(s',x')$, 
the coordinate processes $Z^e$ and $Z^{e'}$ behave as independent Brownian motions up to 
$$
T^{ee'}=\inf\{t\ge s\vee s':Z_t^e-Z_t^{e'}\in\Z\},
$$
after which they continue to move as Brownian motions, but now with a constant separation.
In particular, if $s=s'$ and $x'=x+n$ for some $n\in\Z$, then $T^{ee'}=0$, so $Z^{e'}_t=Z^e_t+n$
for all $t\ge s$, so (\ref{ZDEG1}) holds almost surely.

Let $(s,x)\in E$ and $t,u\in\Q$, with $s\le t<u$.
Consider the event
$$
A=\left\{\sup_{y\in\Q,y<Z_t^{(s,x)}}Z_u^{(t,y)}=Z_u^{(s,x)}=\inf_{y'\in\Q,y'>Z_t^{(s,x)}}Z_u^{(t,y')}\right\}.
$$
Fix $n\in\N$ and set $Y=n^{-1}\lfloor nZ_t^{(s,x)}\rfloor$ and $Y'=Y+1/n$.
Then $Y$ and $Y'$ are $\cF_t$-measurable, $\Q$-valued random variables.
Now $\PP(Y<Z_t^{(s,x)}<Y')=1$ and 
$$
\{Y<Z_t^{(s,x)}<Y'\}\cap\{T^{(t,Y)(t,Y')}\le u\}\sse A.
$$
By the Markov property of Brownian motion, almost surely,
$$
\PP(T^{(t,Y)(t,Y')}\le u|\cF_t)=2\Phi\left(\frac1{n\sqrt{2(u-t)}}\right),
$$
and the right-hand side tends to $1$ as $n\to\infty$.
So, by bounded convergence, we obtain $\PP(A)=1$. 
On taking a countable intersection of such sets $A$ over the possible values
of $s,x,t$ and $u$, we deduce that
conditions (\ref{ZRIGHT}), (\ref{ZFLOW}) and (\ref{ZLEFT}) hold almost surely.

It remains to establish the continuity condition (\ref{ZCADLAG}). 
For a standard Brownian motion $B$ starting from $0$, we have, for $n\ge4$,
$$
\PP\left(\sup_{t\le1}|B_t|>n\right)\le e^{-n^2/2}.
$$
Define, for $\d>0$ and $e=(s,x)\in E$, 
$$
V^e(\d)=\sup_{s\le t\le s+\d^2}|Z^e_t-x|.
$$
Then, by scaling, 
$$
\PP(V^e(\d)>n\d)\le e^{-n^2/2}.
$$
Consider, for each $n\in\N$ the set 
$$
E_n=\{(j2^{-2n},k2^{-n}):j\in\tfrac12\Z\cap[-2^{2n},2^{2n}),k=0,1,\dots,2^n-1\}
$$
and the event
$$
A_n=\bigcup_{e\in E_n}\{V^e(2^{-n})>n2^{-n}\}.
$$
Then $\PP(A_n)\le|E_n|e^{-n^2/2}$, so $\sum_n\PP(A_n)<\infty$, so by Borel--Cantelli, almost surely, 
there is a random $N<\infty$ such that $V^e(2^{-n})\le n2^{-n}$ for all $e\in E_n$, for all $n\ge N$.

Given $\ve>0$, choose $n\ge N$ such that $(4n+2)2^{-n}\le\ve$ and set $\d=2^{-2n-1}$. Then, for all
rationals $s,t\in(-n,n)$ with $0<t-s<\d$ and all rationals $x\in[0,1]$, there exist $e^\pm=(r,y^\pm)\in E_n$ such that 
$$
r\le s<t\le r+2^{-2n},\q  x+n2^{-n}<y^+\le x+(n+1)2^{-n},\q x-(n+1)2^{-n}\le y^-<x-n2^{-n},
$$
Then, $Z^{e^-}_s<x<Z_s^{e^+}$, so
$$
x-\ve\le Z^{e^-}_t\le Z_t^{(s,x)}\le Z_t^{e^+}\le x+\ve.
$$
Hence $\|\Phi_{(s,t]}-\id\|\le\ve$, as required.
\end{proof}

Recall that $\Phi^E$ denotes the common restriction of $\Phi^{E,+}$ and $\Phi^{E,-}$ to $D^\circ_E$.
\begin{proposition}
Let $E=\Q^2$. Then $\Phi^E$ is continuous at $z$ for all $z\in C^\circ_E$.
\end{proposition}
\begin{proof}
Consider a sequence $(z_k:k\in\N)$ in $D^\circ_E$ and suppose that $z_k\to z$ in $D_E$, with $z\in C^\circ_E$.
Set $\phi^k=\Phi^E(z_k)$ and $\phi=\Phi^E(z)$. 
It will suffice to show that, for all $n\in\N$, we have $d_C^{(n)}(\phi^k,\phi)\to0$ as $k\to\infty$.
Given $\ve>0$, choose $\ve'>0$ and $\eta>0$ so that $\ve'+2\eta<\ve$.
Then choose $m\in\N$ so that $\ve'+2\eta+1/m<\ve$ and so that $\|\phi_{ts}-\id\|<\eta$ for all $s,t\in(-n,n)$
with $0<t-s<1/m$. Consider the finite set
$$
F=(m^{-1}\Z\cap[-n,n))\times(m^{-1}\Z\cap[0,1)).
$$
There exists a $K<\infty$ such that, for all $k\ge K$, all $(s_0,x_0)\in F$, and all $t\in(s_0,n]$,
$$
|\phi^{k,+}_{(s_0,t]}(x_0)-\phi_{ts}^+(x_0)|=|\phi^{k,-}_{(s_0,t]}(x_0)-\phi_{ts}^-(x_0)|<\ve'.
$$
For all $s\in[-n,n)$ and all $x\in[0,1)$, there exists $(s_0,x_0)\in F$ such that
$$
s_0\le s<s_0+1/m,\q x_0\le x+\ve'+\eta+1/m<x_0+1/m.
$$
Then
$$
\phi^{k,+}_{(s_0,s]}(x_0)\ge\phi^+_{ss_0}(x_0)-\ve'\ge x_0-\ve'-\eta>x,
$$
so
$$
\phi^{k,+}_{(s_0,t]}(x_0)\ge \phi^{k,+}_{(s,t]}(x),\q t\ge s.
$$
Also, we have
$$
\phi_{ss_0}^+(x_0)\le x_0+\eta\le x+\ve'+2\eta+1/m<x+\ve,
$$
so
$$
\phi^+_{ts_0}(x_0)\le\phi^+_{ts}(x+\ve),\q t\ge s.
$$
Now, for all $t\in(s,n]$,
$$
\phi_{(s_0,t]}^{k,+}(x_0)\le\phi_{ts_0}^+(x_0)+\ve,
$$
so
$$
\phi_{(s,t]}^{k,+}(x)\le\phi_{ts}^+(x+\ve)+\ve.
$$
By a similar argument, for all $t\in(s,n]$,
$$
\phi_{(s,t]}^{k,-}(x)\ge\phi_{ts}^-(x-\ve)-\ve,
$$
so $d_\cD(\phi^k_{(s,t]},\phi_{(s,t]})\le\ve$. Hence $d_C^{(n)}(\phi^k,\phi)\to0$ as $k\to\infty$, as required.
\end{proof}

\subsection{From the coalescing Brownian flow to the Brownian webs}
The coalescing Brownian flow, which provides the limit object for our main result
is a refinement, in certain respects, of Arratia's flow of coalescing Brownian motions,
via the work of T\' oth and Werner. In a series of works, beginning with \cite{FINR}, Fontes, Newman and others
have already provided another such refinement, in fact several, which they call Brownian webs.
They direction taken by Fontes et al. emphasises path properties: 
the Brownian web is conceived as a random element of a space $\cH$ of compact collections of
$\R$-valued paths with specified starting points. 
In our formulation, one does not see so clearly the possible varieties of path, but we find the
state space $C^\circ(\R,\cD)$ a convenient one, with a natural time-reversal map, and just one
candidate for a probability measure corresponding to Arratia's flow.
We now attempt to clarify the relationship between 
the coalescing Brownian flow and one of the Brownian webs.

In \cite{FINR}, the authors make the comment that there is more than
one natural distribution of an $\cH$-valued\footnote{We refer to \cite{FINR} for precise definitions and attempt here only to
give a flavour of their results.} random variable $\cW$ that satisfies the following
two conditions.
\begin{enumerate}
  \item[(i)]
   From any deterministic point $(x,t)$ in space-time, there is almost
   surely a unique path $W_{x,t}$ in $\cW$ starting from $(x,t)$.
  \item[(ii)]
   For any deterministic $n$ and $(x_1, t_1), \ldots, (x_n, t_n)$, the
   joint distribution of $W_{x_1, t_1}, \ldots, W_{x_n, t_n}$ is that
   of coalescing Brownian motions (with unit diffusion constant).
\end{enumerate}
The {\em standard Brownian web} satisfies (i) and (ii) together with a certain
minimality condition.
On the other hand, the {\em forward full Brownian web}, introduced in \cite{FN},
satisfies (i) and (ii) together with a certain maximality condition, subject to a
non-crossing condition.
Also in \cite{FN}, the authors characterize a third object, the {\em full
Brownian web}, which is a random variable on the space $\cH^F$ of
compact collections of paths from $\R\to\R$, and explain how this is naturally related
to the other Brownian webs.

We now describe in detail the space $\cH^F$ of the full Brownian web and give its characterization. 
We then show that there is a natural way to realise the full Brownian web (and hence also the standard 
Brownian web and forward full Brownian web) as a random variable on $C^\circ(\R,\cD)$. 
Define the function $\Phi : [-\infty, \infty]\times\R \rightarrow [0,1]$ by
$$
\Phi(x,t)=\frac{\tanh(x)}{1+|t|}.
$$
Now construct the two metric spaces $(\Pi^F,d)$ and $(\cH^F,d_{\cH^F})$ as follows. 
Let $\Pi^F$ denote the set of functions $f:\R\rightarrow[-\infty,\infty]$ such that $\Phi(f(t),t)$
is continuous, and define
$$
d^F(f_1,f_2)= \sup_{t\in\R}|\Phi(f_1(t),t)-\Phi(f_2(t),t)|.
$$
Then $(\Pi^F,d^F)$ is a complete separable
metric space.
Now let $\cH^F$ denote the set of compact
subsets of $(\Pi^F,d^F)$, with $d_{\cH^F}$ the induced Hausdorff metric
$$
d_{\cH^F}(K_1,K_2)=\sup_{g_1\in K_1}\inf_{g_2\in K_2}d^F(g_1,g_2)\vee
              \sup_{g_2\in K_2}\inf_{g_1\in K_1}d^F(g_1,g_2).
$$
The space $(\cH^F,d_{\cH^F})$ is also complete and separable. Let $\cF_{\cH^F}$ be the Borel $\sigma$-algebra on $\cH^F$.
The full Brownian web is defined in \cite{FN} as follows.
\begin{defn}
A full Brownian web $\bar{\cW}^F$ is any $({\cH^F},{\cF}_{{\cH^F}})$-valued random variable
whose distribution has the following properties.
\begin{itemize}
\item[(a)] Almost surely the paths of $\bar{\cW}^F$ are noncrossing (although
they may touch, including coalescing and bifurcating). 
\item[($b_1$)]
From any deterministic point $(x,t) \in \mathbb{R}^2$,
there is almost surely a unique path $W_{x,t}^F$ passing through
$x$ at time $t$.
\item[($b_2$)]
For any deterministic $m$, $\{ (x_1,t_1), \ldots, (x_m,t_m) \}$,
the joint distribution of the semipaths $\{ W^F_{x_j, t_j}(t), t
\geq t_j, j=1, \dots, m \}$ is that of a flow of coalescing Brownian
motions (with unit diffusion constant).
\end{itemize}
\end{defn}
They show that any two full Brownian webs have the same distribution. We make a variation of this definition, to tie up with our
work, in requiring in ($b_2$) that the semipaths have the law of coalescing Brownian motions {\em on the circle}.

Fix $E$, a countable subset of $\R^2$ containing $\Q^2$. 
We define a map $\theta_{F} : C^\circ(\R,\cD) \rightarrow \cH^F$ as follows.
Let
$$
\Pi = \bigcup_{t \in \R} C([t, \infty), [-\infty, \infty]).
$$
Let $\bar{\cH}$ be the set of all subsets $A$ of $\Pi\cup\Pi^F$ having the the following
noncrossing property
$$
\text{ for all $f,g\in A$ and all $s,t\in\dom(f)\cap\dom(g)$, $f(s)<g(s)$ implies $f(t)\le g(t)$.}
$$
We say that a set $U \in \bar{\cH}$ is {\em maximal} if, 
for any $f\in \Pi\cup\Pi^F$, $U \cup \{ f \} \in \bar{\cH}$ implies $f \in U$.
For $\phi \in C^\circ(\R,\cD)$, define $\th_0(\phi)=\{Z^{e,+}(\phi):e\in E\}$.  
Note that, since $E$ is dense in $\R^2$, if $f,g\in\Pi\cup\Pi^F$ and $\th_0(\phi)\cup\{f\},\th_0(\phi)\cup\{g\}\in\bar{\cH}$,
then also $\th_0(\phi)\cup\{f,g\}\in\bar{\cH}$.
Hence there exists a unique maximal set $\theta_{FF}(\phi)\in\bar{\cH}$ containing $\th_0(\phi)$. Furthermore, this set is independent of the choice of $E$.
Define $\theta_F(\phi) = \theta_{FF}(\phi) \cap \Pi^F$.

This can be viewed as a random variable on $C^\circ(\R,\cD)$ as a consequence of Propositions \ref{thetainh} and \ref{isomet} below. 
We note that, under the measure $\mu_W$, the conditions to be a full Brownian web are almost surely satisfied: 
(a) is immediate from the construction; 
($b_1$) follows from Proposition \ref{FSX}, where condition (i) holds almost surely by Proposition \ref{MPI}, and condition (ii) depends on a countable number of almost sure conditions; 
($b_2$) follows from Theorem \ref{UBPR}. In fact $\theta_{FF}$ defined above is a forward full Brownian web, 
and the standard Brownian web can be constructed similarly, but we omit the details here.

\begin{prop}
\label{thetainh}
The set $\theta_{F}(\phi) \subset \Pi^F$ is compact.
\end{prop}
  \begin{proof}
   Suppose that $f_1, f_2, \ldots \in \theta_{F}(\phi)$. Since paths in
   $\theta_{F}(\phi)$ are noncrossing, there exists a subsequence $n_r$
   such that $f_{n_r}(t)$ is monotone for all $t \in \R$ and so
   there exists some $f:\R \rightarrow [-\infty, \infty]$ such that $f_{n_r}
   \rightarrow f$ pointwise.

   Since $\phi_{ts}(\cdot)$ is periodic with period 1, if $g \in \theta_{F}(\phi)$, then $g+m \in \theta_{F}(\phi)$ for all $m \in \mathbb{Z}$. Suppose $f(s) = \infty$ for some $s \in \mathbb{R}$. By restricting to a further subsequence if necessary, we may assume $f_{n_r}(s) \geq f_{n_{r-1}}(s) + 1$ for all $r$. But then, by noncrossing, $f_{n_r}(t) \geq f_{n_{r-1}}(t) + 1$ for all $t \in \mathbb{R}$. Therefore $f(t) = \infty$ for all $t \in \mathbb{R}$ and $f_{n_r} \rightarrow f$ in $(\Pi^F,d^F)$. A similar argument applies if $f(s)=-\infty$ for some $s \in \mathbb{R}$. Hence we may assume $|f(t)| < \infty$ for all $t \in \mathbb{R}$.

   Note that $d^F(f,g) \leq d^{(n)}(f,g) \vee 2/(n+1)$ where
   $$
   d^{(n)}(f,g) = \sup_{-n < t < n} |f(t) - g(t)|.
   $$
   Given $\epsilon > 0$, there exist $n > 2\epsilon^{-1}$, and $-n = a_0 < \cdots < a_M=n$ such that
   $\| \phi^+_{s a_k} - \id \| < \frac{\epsilon}{4}$ for all $a_k \leq s \leq
   a_{k+1}$, $k = 0, \ldots, M-1$. Pick $N$ sufficiently large
   that if $n_r > N$, $|f_{n_r}(a_k)-f(a_k)| <
   \frac{\epsilon}{4}$ for $k = 1, \ldots, M-1$. Then for $a_k
   \leq s < a_{k+1}$,
   $$
   |f_{n_r}(s) - f(s)| \leq |\phi^+_{s a_{k}}(f(a_k) +
   \tfrac{\epsilon}{4}) - \phi^+_{s a_{k}}(f(a_k) - \tfrac{\epsilon}{4})|
   < \epsilon.
   $$
   Hence, $f_{n_r} \rightarrow f$ in $(\Pi^F,d^F)$. Therefore, $f$
   is continuous and $\theta_{FF}(\phi) \cup \{ f \}$ is noncrossing,
   and so $f \in \theta_F (\phi)$ proving compactness.
\end{proof}

\begin{prop}
\label{FSX}
For every $(x,s) \in \mathbb{R}^2$, there exists some $f \in \theta_F (\phi)$ with $f(s)=x$. Moreover, $f$ is unique if the following two conditions hold.
\begin{enumerate}
 \item[(i)] $\phi^+_{ts}(x) = \phi^+_{ts}(x)$ for all $t \geq s$;
 \item[(ii)] $\phi^+_{st}(y) \neq x$ for all $(y,t) \in \mathbb{Q}^2$ with $t < s$.
\end{enumerate}
\end{prop}
\begin{proof}
 Define a function $f:\mathbb{R} \rightarrow \mathbb{R}$ by
 $$
 f(t) = \begin{cases}
         \inf \{ y : \phi_{st}^+(y) > x \} &\text{if $t < s$} \\
         \phi^+_{ts}(x) &\text{if $t \geq s$}.
        \end{cases}
 $$
We first show that $f$ is continuous. Clearly this is the case on $[s, \infty)$. Suppose $t < u \leq s$. Pick a sequence $y_n \downarrow f(t)$. By the weak flow property of $\phi$, $\phi^+_{su}(\phi^+_{ut}(y_n)) \geq \phi^+_{st}(y_n)>x$ and so $\phi^+_{ut}(y_n) \geq f(u)$ for all $n$. Letting $n \rightarrow \infty$ gives $\phi^+_{ut}(f(t)) \geq f(u)$. Similarly, using $\inf \{ y : \phi_{st}^+(y) > x \}=\sup \{ y : \phi_{st}^-(y) \leq x \}$, $\phi^-_{ut}(f(t)) \leq f(u)$. Now given $\epsilon > 0$, there exists some $\delta > 0$ such that $\| \phi_{ut} - \id \| < \epsilon$ for all $0 < u -t < \delta$. But then $|f(t) - f(u)| \leq |f(t) - \phi^-_{ut}(f(t))|\wedge|f(t) - \phi^+_{ut}(f(t))|< \epsilon$, proving continuity. 
Also, by construction, $\th_0(\phi)\cup\{f\}\in\bar{\cH}$ and so $f \in \theta_F (\phi)$ as required.

For uniqueness, suppose $f_1, f_2 \in \theta_F (\phi)$ with $f_1(s)=f_2(s)=x$. By the noncrossing property, for all $t \geq s$, $\phi_{ts}^-(x) \leq f_i(t) \leq \phi_{ts}^+(x)$. Condition (i) therefore implies $f_1(t)=f_2(t)$ for all $t \geq s$. If $t < s$, then $\sup \{ y : \phi_{st}^-(y) < x \} \leq f_i(t) \leq \inf \{ y : \phi_{st}^+(y) > x \}$. If $\inf \{ y : \phi_{st}^+(y) > x \} - \sup \{ y : \phi_{st}^-(y) < x \} > 0$, then by continuity, there exist rationals $(y,u)$ with $\phi^+_{st}(y) = x$. Hence (ii) ensures $f_1(t)=f_2(t)$ for all $t < s$.
\end{proof}

  \begin{prop}
  \label{isomet}
   The function $\theta_{F} : C^\circ(\R,\cD) \rightarrow \cH^F$ is continuous.
  \end{prop}
  \begin{proof}
   A sequence $\phi^m \rightarrow \phi$ in $(C^\circ(\R,\cD), d_C)$ if and only if $d^{(n)}_C(\phi^m, \phi) \rightarrow 0$ for all $n \in \mathbb{N}$. Since $d^F(f,g) \leq d^{(n)}(f,g) \vee 2/(n+1)$, where $d^{(n)}$ is defined in the previous proposition, in order to show that $\theta_{F}$ is continuous, it is enough to show that for every $n$, $d_{\cH^F}^{(n)}(\theta_{F}(\phi^1),\theta_{F}(\phi^2)) \leq d^{(n)}_C(\phi^1, \phi^2)$ for all $\phi^1, \phi^2 \in C^\circ(\R,\cD)$, where
   $$
   d^{(n)}_{\cH^F}(K_1,K_2)=\sup_{g_1\in K_1}\inf_{g_2\in K_2}d^{(n)}(g_1,g_2)\vee
              \sup_{g_2\in K_2}\inf_{g_1\in K_1}d^{(n)}(g_1,g_2).
   $$
   
Suppose $\phi^1, \phi^2 \in C^\circ(\R,\cD)$. By compactness, there exist $f_1 \in \theta_{F}(\phi^1)$, $f_2 \in \theta_{F}(\phi^2)$ such that $d^{(n)}(f_1,f_2) = d^{(n)}_{\cH^F}(\theta_{F}(\phi^1), \theta_{F}(\phi^2))$. Without loss of generality, suppose that $f_2$ is
chosen so that $d^{(n)}(f_1,f_2) = \inf_{g \in \theta_{F}(\phi^2)} d^{(n)}(f_1,g) = f_2(s) - f_1(s)$ for some $s \in \mathbb{R}$. Let $f_1(s)=x$ and $f_2(s)=y$. By continuity, there exists some $t > s$ such that $f_1(t) - \phi^-_{ts}(y) = d^{(n)}(f_1,f_2)$, otherwise it would be possible to pick some $g \in \theta_{F}(\phi^2)$ with $d^{(n)}(f_1,g) < d^{(n)}(f_1,f_2)$.
Since $f_1(t) \in [\phi^{1,-}_{ts}(x), \phi^{1,+}_{ts}(x)]$,
   $$
   \phi^{2,-}_{ts}(y) - \phi^{1,+}_{ts}(x) \leq x-y \leq \phi^{2,+}_{ts}(y) -
   \phi^{1,-}_{ts}(x).
   $$
   Hence there exists some
   $$
   u \in \left [ \frac{1}{2}(x + \phi^{1,-}_{ts}(x)), \frac{1}{2}(x +
   \phi^{1,+}_{ts}(x)) \right ] \cap \left [ \frac{1}{2}(y +
   \phi^{2,-}_{ts}(y)), \frac{1}{2}(y + \phi^{2,+}_{ts}(y)) \right ].
   $$
   Therefore,
   \begin{eqnarray*}
     d_{\cH^F}(\theta_F(\phi^1), \theta_F(\phi^2))
       &=& |x-y| \\
       &=& |\phi^{1\times}_{ts}(u) -\phi^{2\times}_{ts}(u)| \\
    &\leq& \| \phi^{1\times}_{ts} - \phi^{2\times}_{ts} \| \\
    &\leq& d^{(n)}_{C}(\phi^1, \phi^2),
   \end{eqnarray*}
   as required.
  \end{proof}

\markright{\sc{Bibliography}}

\def\cprime{$'$}

\end{document}